\documentclass[reqno,tbtags,a4paper,12pt]{amsart}

\usepackage[in]{fullpage}
\usepackage{makecell}
\usepackage[matrix,arrow,curve]{xy}
\usepackage{tikz}
\usetikzlibrary{math}
\usetikzlibrary{shapes}
\usetikzlibrary{bending}
\usetikzlibrary{intersections}

\usepackage{amsmath,amssymb,amsthm,amsfonts,amscd,mathtools,multirow,afterpage}
\usepackage{graphics,graphicx}

\usepackage{mathtools}

\newcommand{\Z}{\mathbb{Z}}

\newcommand{\R}{\mathbb{R}}

\newcommand{\RP}{\mathbb{RP}}
\newcommand{\X}{\mathbb{X}}
\newcommand{\E}{\mathbb{E}}
\newcommand{\C}{\mathbb{C}}
\newcommand{\bS}{\mathbb{S}}
\newcommand{\Lob}{\Lambda}

\newcommand{\bell}{\boldsymbol{\ell}}

\newcommand{\ba}{\mathbf{a}}
\newcommand{\bb}{\mathbf{b}}
\newcommand{\bc}{\mathbf{c}}
\newcommand{\br}{\mathbf{r}}
\newcommand{\bn}{\mathbf{n}}
\newcommand{\bm}{\mathbf{m}}
\newcommand{\bu}{\mathbf{u}}
\newcommand{\bo}{\mathbf{o}}
\newcommand{\bv}{\mathbf{v}}
\newcommand{\bw}{\mathbf{w}}

\newcommand{\bp}{\mathbf{p}}

\newcommand{\Vor}{V_{\mathrm{or}}}

\newcommand{\dn}{\mathop{\mathrm{dn}}\nolimits}
\newcommand{\cn}{\mathop{\mathrm{cn}}\nolimits}
\newcommand{\sn}{\mathop{\mathrm{sn}}\nolimits}
\newcommand{\dist}{\mathop{\mathrm{dist}}\nolimits}
\newcommand{\sgn}{\mathop{\mathrm{sgn}}\nolimits}
\newcommand{\area}{\mathop{\mathrm{area}}\nolimits}
\newcommand{\volume}{\mathop{\mathrm{vol}}\nolimits}

\newcommand{\CC}{\mathcal{C}}

\newcommand{\CV}{\mathcal{V}}

\newtheorem{theorem}{Theorem} [section]
\newtheorem{propos}[theorem] {Proposition}

\newtheorem{cor}[theorem] {Corollary}

\newtheorem{conj}[theorem] {Conjecture}

\newtheorem{question}[theorem] {Question}
\theoremstyle{definition}

\newtheorem{remark}[theorem]{Remark}
\newtheorem{defin}[theorem]{Definition}

\numberwithin{equation}{section}

\author{Alexander A. Gaifullin}

\address{Steklov Mathematical Institute of Russian Academy of Sciences, Moscow, Russia}
\address{Lomonosov Moscow State University, Moscow, Russia}
\email{agaif@mi-ras.ru}

\thanks{}

\title{Exotic spherical flexible octahedra and counterexamples to the Modified Bellows Conjecture}

\date{}



\begin{document}
\begin{abstract}
In 2014 the author showed that in the three-dimensional spherical space, alongside with three classical types of flexible octahedra constructed by Bricard, there exists a new type of flexible octahedra, which was called exotic. In the present paper we give a geometric construction for exotic flexible octahedra, describe their configuration spaces, and calculate their volumes. We show that the volume of an exotic flexible  octahedron is nonconstant during the flexion, and moreover the volume remains nonconstant if we replace any set of vertices of the octahedron with their antipodes. So exotic flexible octahedra are counterexamples to the Modified Bellows Conjecture proposed by the author in 2015.
\end{abstract}

\maketitle

\section{Introduction}

\subsection{Flexible octahedra}

Let $\X^3$ be one of the three three-dimensional spaces of constant curvature, i.\,e., the Euclidean space~$\E^3$, or the Lobachevsky space~$\Lob^3$, or the sphere~$\bS^3$. By \textit{octahedron} we always mean a (possibly self-intersecting) polyhedral surface in~$\X^3$ that has the combinatorial type of a regular octahedron. This means that an octahedron has $6$ vertices $\ba_1$, $\ba_2$, $\ba_3$, $\bb_1$, $\bb_2$, and~$\bb_3$, $12$ edges spanned by all pairs of vertices except for the three pairs~$\{\ba_i,\bb_i\}$, and $8$ triangular faces spanned by all triples of vertices that are pairwise connected by edges. We always assume that all faces are non-degenerate, that is, the vertices of every face do not lie on a line. Suppose that the vertices of an octahedron move continuously in~$\X^3$ so that every face of the octahedron remains congruent to itself. The obtained continuous family of octahedra~$P(u)$ is called a \textit{flex} unless it is induced by a continuous family of isometries of the whole space~$\X^3$.

To avoid certain degenerate situations, we need the following concept of an essential octahedron. An octahedron is said to be \textit{essential} if every dihedral angle of it is neither zero nor straight. A flexible octahedron~$P(u)$ is said to be \textit{essential} if the octahedron~$P(u)$ is essential for all but a finite number values of~$u$.

Bricard~\cite{Bri97} (cf.~\cite{Ben12}) classified (essential) flexible octahedra in Euclidean $3$-space~$\E^3$. He showed that there are three types of them:
\smallskip

\textsl{Type 1: line-symmetric flexible octahedra}. For every $i=1,2,3$, the combinatorially opposite vertices~$\ba_i$ and~$\bb_i$ are symmetric to each other about a line~$L$ (not depending on~$i$).
\smallskip

\textsl{Type 2: plane-symmetric flexible octahedra}. For two values of the index~$i$ (say, for $i=1,2$) the combinatorially opposite vertices~$\ba_i$  and~$\bb_i$ are symmetric to each other about a plane~$\Pi$ (not depending on~$i$), and the remaining two vertices~$\ba_3$ and~$\bb_3$ lie on~$\Pi$.
\smallskip

\textsl{Type 3: skew flexible octahedra}. These flexible octahedra are unsymmetric. The most important for us property of them is that the tangents of the halves of all their dihedral angles remain either directly or inversely proportional to each other during the flexion. This, in particular, implies that when some dihedral angle is either zero or straight, then all other dihedral angles are either zero or straight, too. Hence, a skew flexible octahedron has two flat positions.
\smallskip

Bricard's classification relies upon the study of algebraic relations between the variables $t_i=\tan(\varphi_i/2)$, where $\varphi_1$, $\varphi_2$, and~$\varphi_3$ are the oriented dihedral angles (defined analogously to Section~\ref{section_Bricard}) of an octahedron at the edges~$\ba_2\ba_3$, $\ba_3\ba_1$, and~$\ba_1\ba_2$, respectively. Bricard proved that each pair of the variables~$t_1$, $t_2$, and~$t_3$ satisfies a biquadratic relation of the form
\begin{equation}\label{eq_rel}
 A_{ij}t_i^2t_j^2+B_{ij}t_i^2+2C_{ij}t_it_j+D_{ij}t_j^2+E_{ij}=0.
\end{equation}
Then he analyzed all possible cases regarding the reducibility of these relations. One obtains a flexible octahedron if and only if the algebraic variety in~$\RP^1\times\RP^1\times\RP^1$ given by the three equations~\eqref{eq_rel} with $(i,j)=(1,2)$, $(2,3)$, and~$(3,1)$ contains a one-dimensional irreducible component~$\Gamma$. We denote by~$\R(\Gamma)$ the field of rational functions on~$\Gamma$, and by~$\R(t_i)$ (respectively,~$\R(t_i,t_j)$) the subfield of~$\R(\Gamma)$ consisting of rational functions in~$t_i$ (respectively, in~$t_i$ and~$t_j$). For a subfield~$F$ of a field~$E$, we denote by~$|E/F|$ the degree of~$E$ over~$F$.  In modern language, Bricard's classification can be formulated as follows.
\smallskip

(1) A flexible octahedron is line-symmetric if and only if any pair of variables~$t_i$ and~$t_j$ is neither directly nor inversely proportional to each other on~$\Gamma$. This can happen in two subcases:

(1a) $\R(\Gamma)=\R(t_1,t_2)=\R(t_2,t_3)=\R(t_1,t_3)$ and $|\R(\Gamma)/\R(t_i)|=2$ for all~$i$.

(1b) After a permutation of indices~$1$, $2$, and~$3$, we have that $\R(\Gamma)=\R(t_1)=\R(t_2,t_3)$ and $|\R(\Gamma)/\R(t_2)|=|\R(\Gamma)/\R(t_3)|=2$.
\smallskip

(2) A flexible octahedron is plane-symmetric if and only if two variables, say~$t_1$ and~$t_2$, are either directly or inversely proportional to each other on~$\Gamma$, and $t_3$ is neither directly nor inversely proportional to them. Then $\R(t_1)=\R(t_2)$, $\R(\Gamma)=\R(t_1,t_3)=\R(t_2,t_3)$. This can happen in three subcases:

(2a) $|\R(\Gamma)/\R(t_1)|=|\R(\Gamma)/\R(t_3)|=2$.

(2b) $\R(\Gamma)=\R(t_1)$ and $|\R(\Gamma)/\R(t_3)|=2$.

(2c) $\R(\Gamma)=\R(t_3)$ and $|\R(\Gamma)/\R(t_1)|=2$.
\smallskip

(3) A flexible octahedron is skew if and only if every pair of variables~$t_i$ and~$t_j$ are either directly or inversely proportional to each other on~$\Gamma$. Then $\R(\Gamma)=\R(t_i)$ for all~$i$.
\smallskip

Moreover, the listed cases~(1a), (1b), (2a), (2b), (2c), and~(3) exhaust all possibilities for essential  flexible octahedra in~$\E^3$. Note also that the curve~$\Gamma$ is elliptic in cases~(1a) and~(2a) and rational in cases~(1b), (2b), (2c), and~(3).

Stachel~\cite{Sta06} showed that all the three types of Bricard's flexible octahedra exist in the Lobachevsky space~$\Lambda^3$ as well. Moreover, his arguments immediately carry over to the octahedra in~$\bS^3$.

If a relation of the form~\eqref{eq_rel} is irreducible, then it can be interpreted as the addition law for Jacobi's elliptic functions. This observation is very old and natural in theory of elliptic functions. In theory of flexible polyhedra it was first used by Izmestiev~\cite{Izm15,Izm17} to study flexible quadrilaterals and flexible Kokotsakis polyhedra. The author~\cite{Gai14c} has used the same approach to obtain a complete classification of flexible cross-polytopes (which are high-dimensional analogues of octahedra) in all spaces~$\E^n$, $\Lob^n$, and~$\bS^n$, and write explicit parametrizations for their flexions.

In the case of~$\E^3$ the results of~\cite{Gai14c} recover Bricard's classification from the point of view of elliptic functions. In particlular, in the most complicated case~(1a), the following parametrization of~$\Gamma$ was obtained:
\begin{equation}\label{eq_dn_param}
t_i = \lambda_i\dn(u-\sigma_i,k), \qquad i=1,2,3,
\end{equation}
where $k$ is an elliptic modulus, $\sigma_1$, $\sigma_2$, and~$\sigma_3$ are pairwise distinct phases, $\lambda_1$, $\lambda_2$, and~$\lambda_3$ are nonzero coefficients, and $u$ is a parameter of the flexion. Note that the elliptic modulus~$k$ may not be real. Selection of the cases of real~$t_1$, $t_2$, and~$t_3$ was made in~\cite[Sect.~7]{Gai14c}. In each of these cases, the standard transformation formulae for Jacobi's elliptic functions (see~\cite[Sect.~13.22]{BaEr55}) allow to rewrite the parametrization through Jacobi's elliptic functions with modulus in~$(0,1)$.

In the case of~$\Lob^3$ it was shown in~\cite{Gai14c} that the classification is the same as in the Euclidean case, namely, any essential flexible octahedron falls into one of the cases (1a), (1b), (2a), (2b), (2c), and~(3). Hence it belongs to one of Bricard's three types of flexible octahedra, which in the Lobachevsky case were constructed by Stachel~\cite{Sta06}. Note that in the case~(1a) we again have a parametrization of the form~\eqref{eq_dn_param}.

In the case of~$\bS^3$ two new phenomena occur. First, in the spherical case there is the following operation of \textit{replacing a vertex with its antipode}. Suppose that $\ba_i(u)$, $\bb_i(u)$, $i=1,2,3$, is a parametrization of a flexible octahedron. Then all edge lengths of the octahedron are constant. Consider any vertex, say~$\ba_1(u)$, and replace it with its antipode~$-\ba_1(u)$. Then all edge lengths of the obtained octahedron with vertices $-\ba_1(u)$, $\ba_2(u)$, $\ba_3(u)$, $\bb_1(u)$, $\bb_2(u)$, $\bb_3(u)$ are also constant, so we obtain a new well-defined flexible octahedron. Therefore, starting from one of the three types of Bricard's octahedra, one can obtain new flexible octahedra by replacing some of vertices with their antipodes. As a result, not every flexible octahedron in cases~(1a) and~(1b) is line-symmetric and not every flexible octahedron in cases~(2a), (2b), and~(2c) is plane-symmetric. It is only true that every flexible octahedron in cases~(1a) and~(1b) (respectively, (2a), (2b), and~(2c)) becomes line-symmetric (respectively, plane-symmetric) after replacing some of its vertices with their antipodes.

Second, in~$\bS^3$ there is a new type of flexible octahedra, which were called \textit{exotic flexible octahedra} in~\cite{Gai14c}. This type corresponds to the following new possibility for the curve~$\Gamma$, which cannot occur for octahedra in~$\E^3$ or~$\Lob^3$:
\smallskip

(4) $\R(t_3)=\R(t_1)\cap\R(t_2)$, $|\R(\Gamma)/\R(t_3)|=4$, and
$$
|\R(\Gamma)/\R(t_1)|=|\R(\Gamma)/\R(t_2)|=|\R(t_1)/\R(t_3)|=|\R(t_2)/\R(t_3)|=2.
$$

In this case the curve $\Gamma$ is elliptic. It can be parametrized in Jacobi's elliptic functions with an elliptic modulus $k\in(0,1)$. There are three subcases, which were called in~\cite{Gai14c} \textit{exotic flexible octahedra of the first, second, and third kind}, respectively. Up to swapping the indices~$1$ and~$2$, the parametrizations in these three subcases are as follows:
\begin{align}
 &\text{1st kind:}&
 t_1&=\lambda_1\dn u,&
 t_2&= \lambda_2\dn\left(u-\frac{K}2\right),&
 t_3&=\lambda_3\left(\dn u+\frac{k'}{\dn u}\right),\label{eq_1kind}\\
 &\text{2nd kind:}&
 t_1&=\lambda_1\dn u,&
 t_2&= \frac{\lambda_2\cn\left(u-K/2\right)}{\sn\left(u-K/2\right)},&
 t_3&=\lambda_3\left(\dn u-\frac{k'}{\dn u}\right),\label{eq_2kind}\\
 &\text{3rd kind:}&
 t_1&=\frac{\lambda_1\cn u}{\sn u},&
 t_2&= \frac{\lambda_2\cn\left(u-K/2\right)}{\sn\left(u-K/2\right)},&
 t_3&=\lambda_3\left(\frac{\cn u}{\sn u}-\frac{k'\sn u}{\cn u}\right).\label{eq_3kind}
\end{align}
Here $k'=\sqrt{1-k^2}$ is the complementary elliptic modulus,
$$
K=\int_0^1\frac{dx}{\sqrt{(1-x^2)(1-k^2x^2)}}
$$
is the real quarter-period, see~\cite[Chapter~XIII]{BaEr55}, and $\lambda_1$, $\lambda_2$, and~$\lambda_3$ are nonzero constants.

\begin{remark}\label{remark_t}
To simplify the formulae and reduce the number of cases in~\cite{Gai14c} the tangents of the halves of internal or external dihedral angles with one or another sign were chosen as parameters~$t_i$. Therefore, the parameters~$t_i$ in formulae~\eqref{eq_1kind}--\eqref{eq_3kind} may differ from the tangents of the halves of positively oriented dihedral angles by transformations of the form $t\mapsto\pm t^{\pm 1}$ and permutation of indices~$1,2,3$. Note that the transformations of the form $t\mapsto\pm t^{\pm 1}$ do not violate the above condition~(4).
\end{remark}

\begin{remark}\label{remark_kind}
 In fact, in~\cite{Gai14c} the author classified not flexible octahedra, but flexible octahedra with a distinguished face. This means that it has been shown that any flexible octahedron~$P$ falls into one of the types (1a), (1b), (2a), (2b), (2c), (3), and~(4) if the parameters~$t_1$, $t_2$, and~$t_3$ are taken to be the tangents of the halves of the dihedral angles at the edges of some face~$F$ of the octahedron, but the question of whether the type of the flexible octahedron depends on the choice of~$F$ has not been studied. However, it is actually not hard to show that the type does not depend on the choice of~$F$. In the case of type~(4), which is the only one of interest to us in the present paper, we will prove this, see Proposition~\ref{propos_44}. On the contrary, any exotic flexible octahedron (i.\,e. flexible octahedron of type~(4)) can be made into any of the three kinds~\eqref{eq_1kind}--\eqref{eq_3kind} by an appropriate choice of the face~$F$, see Proposition~\ref{propos_kinds} for more detail.
\end{remark}

In~\cite{Gai14c} the author has studied exotic flexible octahedra  from only an algebraic point of view. Parametrizations~\eqref{eq_1kind}--\eqref{eq_3kind} were obtained, but no geometric properties of these octahedra were studied. The aim of the present paper is to give a simple geometric construction of exotic flexible octahedra, and to study their geometric properties.

\subsection{Modified Bellows Conjecture}
Recall a general definition of a flexible polyhedron in~$\X^3$.

\begin{defin}
 Suppose that $K$ is an oriented closed two-dimensional simplicial manifold. A \textit{polyhedron} of combinatorial type~$K$ is a mapping $P\colon K\to \X^3$ whose restriction to any $2$-simplex of~$K$ is a linear (in the case $\X^3=\E^3$) or quasi-linear (in the cases $\X^3=\bS^3$ and $\X^3=\Lambda^3$) homeomorphism onto a non-degenerate triangle in~$\X^3$. Here a map $P$ of a $2$-simplex~$[v_0,v_1,v_2]$ to either~$\bS^3$ or~$\Lambda^3$ is said to be \textit{quasi-linear} if
 $$
 P(\lambda_0v_0+\lambda_1v_1+\lambda_2v_2)=\frac{\lambda_0P(v_0)+\lambda_1P(v_1)+\lambda_2P(v_2)}{|\lambda_0P(v_0)+\lambda_1P(v_1)+\lambda_2P(v_2)|},
 $$
 where $\lambda_0$, $\lambda_1$, and~$\lambda_2$ are barycentric coordinates and the sphere~$\bS^3$ (respectively, the Lobachevsky space~$\Lambda^3$) is identified with the standard vector model  of it in the Euclidean space~$\E^4$ (respectively, the pseudo-Euclidean space~$\E^{1,3}$).

 Now, suppose that the vertices of a polyhedron move continuously in~$\X^3$ so that all edge lengths remain constant and hence every face of the polyhedron remains congruent to itself. The obtained continuous family of polyhedra $P(u)\colon K\to\X^3$ is called a \textit{flex} unless it is induced by a continuous family of isometries of the whole space~$\X^3$.
\end{defin}

We denote by $\Vor(\bp_0,\bp_1,\bp_2,\bp_3)$ the oriented volume of a tetrahedron~$\bp_0\bp_1\bp_2\bp_3$ in~$\X^3$.

\begin{defin}
 The (\textit{generalized}) \textit{oriented volume} of a polyhedron $P\colon K\to \X^3$ is, by definition, the number
 \begin{equation*}
  \Vor(P)=\sum_{[v_0,v_1,v_2]\in K}\Vor\bigl(\bo,P(v_0),P(v_1),P(v_2)\bigr),
 \end{equation*}
 where the sum is taken over all positively oriented $2$-simplices of~$K$ and $\bo\in\X^3$ is a point chosen so that all tetrahedra~$\bo P(v_0)P(v_1)P(v_2)$ are non-degenerate. It is a standard fact that in the cases $\X^3=\E^3$ and~$\X^3=\Lambda^3$ the oriented volume $\Vor(P)$ is well defined, that is, independent of the choice of~$\bo$. For $\X^3=\bS^3$ the oriented volume~$\Vor(P)$ is well defined only up to adding $k\cdot\volume(\bS^3)=2\pi^2k$, where $k\in\Z$. So in this case we will always consider $\Vor(P)$ as an element of~$\R/2\pi^2\Z$.
\end{defin}

The following result was conjectured by Connelly~\cite{Con80} and proved by Sabitov~\cite{Sab96} (see also~\cite{CSW97,Sab98a,Sab98b}).

\begin{theorem}[Bellows Conjecture = Sabitov's theorem]
 The generalized oriented volume of any flexible polyhedron in~$\E^3$ is constant during the flexion.
\end{theorem}

Certainly, it is natural to ask whether the Bellows Conjecture (that is, the constancy of the oriented volume) holds for flexible polyhedra in other spaces of constant curvature~$\X^n$,  where $n\ge 3$. Several positive results were obtained by the author of the present paper. Namely, the Bellows Conjecture is true:
\begin{itemize}
 \item for any flexible polyhedron in~$\E^n$, where $n\ge 4$, see~\cite{Gai14a,Gai14b},
 \item for any bounded flexible polyhedron in~$\Lambda^{2k+1}$, where $k\ge 1$, see~\cite{Gai15a},
 \item for any flexible polyhedron with sufficiently small edge lengths in an arbitrary space of constant curvature~$\X^n$, where $n\ge 3$. More precisely, for any combinatorial type~$K$ of flexible polyhedra in~$\X^n$, there exists an $\varepsilon>0$ such that the volume of any flexible polyhedron $P(u)\colon K\to\X^n$ is constant, provided that all lengths of edges do not exceed~$\varepsilon$, see~\cite{Gai17}.
\end{itemize}

On the other hand, it is known that the Bellows Conjecture is false in the spherical case even if we restrict ourselves to considering polyhedra contained in an open hemisphere~$\bS^n_+$. The first example of a flexible polyhedron in~$\bS^3_+$ with a non-constant volume was constructed by Alexandrov~\cite{Ale97}. Later, for any $n\ge 3$, the author~\cite{Gai15b}  constructed a non-self-intersecting flexible cross-polytope in~$\bS^n_+$ with non-constant volume.

However, due to the fact that in the spherical case we have the operation of replacing a vertex with its antipode, we can look at the Bellows Conjecture from the following angle. For each spherical flexible polyhedron $P(u)$ with $m$ vertices, we can construct $2^m$ spherical flexible polyhedra~$P_1(u),\ldots,P_{2^m}(u)$ by replacing different sets of vertices of~$P(u)$ with their antipodes. So we can consider the $2^m$ oriented volumes $\Vor\bigl(P_1(u)\bigr),\ldots,\Vor\bigl(P_{2^m}(u)\bigr)$ as $2^m$ volumes corresponding to the original polyhedron~$P(u)$. It is natural to ask whether it is true that at least one of these $2^m$ volumes is constant during the flexion. This was conjectured by the author in~\cite{Gai15b}.

\begin{conj}[Modified Bellows Conjecture]
 Suppose that $P(u)$ is a flexible polyhedron in~$\bS^n$, where $n\ge 3$. Then we can replace some vertices of~$P(u)$ with their antipodes so that the oriented volume of the obtained flexible polyhedron~$P'(u)$ will be constant during the flexion.
\end{conj}

Note that this conjecture is true for all counterexamples to the usual Bellows Conjecture constructed in~\cite{Ale97,Gai15b}. Nevertheless, in the present paper we will show that the Modified Bellows Conjecture is false for flexible polyhedra in~$\bS^3$. Namely, we will obtain the following result.

\begin{theorem}\label{thm_volume}
 Any exotic flexible octahedron in~$\bS^3$ is a counterexample to the Modified Bellows Conjecture.
\end{theorem}

A rigorous definition of an exotic flexible octahedron will be provided in Section~\ref{section_Bricard}, see Definition~\ref{defin_exotic}.  A more precise statement of Theorem~\ref{thm_volume} will be given in Section~\ref{section_volume}, see Theorem~\ref{thm_volume_precise}.

\subsection{Structure of the paper} In Section~\ref{section_constr}, we give a geometric construction of exotic flexible octahedra in~$\bS^3$. In Section~\ref{section_Bricard}, we have collected the known facts about flexible octahedra and their configuration spaces that we need. In Section~\ref{section_alg_geo} we establish the equivalence between the algebraic and geometric descriptions of exotic flexible octahedra. Namely, we prove that a flexible octahedron in~$\bS^3$ is given by the geometric construction from Section~\ref{section_constr} if and only if it falls into the case~(4) in the above algebraic classification. Section~\ref{section_config_exotic} contains a description of the configuration space of an exotic flexible octahedron. In Section~\ref{section_volume}, we calculate the oriented volume of an exotic flexible octahedron and prove Theorem~\ref{thm_volume}. Note that the proof of this theorem relies only on the geometric construction of exotic flexible octahedra from Section~\ref{section_constr} and does not use their explicit parametrization in elliptic functions obtained in~\cite{Gai14c}. Therefore, Sections~\ref{section_config_exotic} and~\ref{section_volume} can be read independently of Section~\ref{section_alg_geo}, which establishes the equivalence of the geometric and algebraic descriptions. Moreover, by replacing the arguments related to the analyticity of the volume function with direct, albeit somewhat cumbersome, calculations, one can make the proof of Theorem~\ref{thm_volume} in Section~\ref{section_volume} independent of the content of Section~\ref{section_config_exotic} as well, see Remark~\ref{remark_direct}. Finally, Section~\ref{section_conclusion} contains some remarks and open questions.

\section{Geometric construction of an exotic flexible octahedron}\label{section_constr}

We realize $\bS^3$ as the unit sphere centered at the origin in the Euclidean space~$\E^4$ with coordinates $x_1,x_2,x_3,x_4$. In describing geometric constructions in~$\bS^3$, we will use the terminology of spherical geometry. In particular, great circles and great spheres will be called lines and planes, respectively. We denote by $\dist(\bu,\bv)$ the spherical distance between two points $\bu,\bv\in\bS^3$, and by~$\angle \bu\bv\bw$ the angle between the arcs of great circles~$\bv\bu$ and~$\bv\bw$. We denote the scalar product in~$\E^4$ by~$\langle\cdot,\cdot\rangle$.

Let $L_1$ be the line (great circle) in~$\bS^3$ given by $x_3=x_4=0$ and $L_2$ the line given by $x_1=x_2=0$. Then the distance in $\bS^3$ from every point of~$L_1$ to every point of~$L_2$ is equal to~$\pi/2$. Moreover, any point $\br\in\bS^3\setminus (L_1\cup L_2)$ can be uniquely written as
$$
\cos \theta\cdot \br_1+\sin \theta\cdot \br_2,\qquad \br_1\in L_1,\ \br_2\in L_2,\ \theta\in\left(0,\frac{\pi}{2}\right).
$$
Note that rotation of~$\bS^3$ around~$L_1$ keeps $L_2$ invariant, and vice versa.

We consider an octahedron $\ba_1\ba_2\ba_3\bb_1\bb_2\bb_3$ in~$\bS^3$ such that
\smallskip

(1) $\ba_1,\bb_1\in L_1$ and $\ba_2,\bb_2\in L_2$,
\smallskip

(2) the vertices $\ba_3$ and $\bb_3$ are symmetric to each other with respect to~$L_1$, that is, are obtained from each other by the rotation by~$\pi$ around~$L_1$,
\smallskip

(3) the faces of the octahedron are non-degenerate, in particular, $\bb_1\ne\pm\ba_1$, $\bb_2\ne\pm\ba_2$, and $\ba_3,\bb_3\notin L_1\cup L_2$.
\smallskip

\begin{figure}
 \begin{tikzpicture}[scale =1.3]

\path [name path = vert2] (-.1,-1.35)--(-.1,5);
\path [name path = vert3] (-.625,0) -- (-.625,5);
\path [name path = vert4] (.75,0) -- (.75,5);
\path [name path = vert5] (-2,-1.5) -- (-2,0);
\draw [line width=1pt, name path =ell] ellipse (3 and 1.5);
  \fill [white] (-.15,1.35) rectangle +(.3,.3);
\draw [line width=1pt,name path = vert1] (0,-1.35)--(0,5);

  \path (0,-.5) coordinate (a1) {}
        (0,2.5) coordinate (b1) {}
        (1.5,4.5) coordinate (b3) {}
          edge [draw,blue,bend left=10] (a1)
          edge [draw,violet,bend right=15] (b1)
        (-2.233,1) coordinate (a2) {}
          edge [draw,cyan,bend left=20,dashed] (b3)
          edge [draw,red,bend right=10] (a1)
          edge [draw,red,bend left=25] (b1)
        (2.827,-.5) coordinate (b2) {}
          edge [draw,red,bend left=10] (a1)
          edge [draw,red,bend right=25] (b1)
          edge [draw,green,bend right=10,dashed] (b3)
        (-1.25,3) coordinate (a3) {}
          edge [draw,blue,bend right=10] (a1)
          edge [draw,violet,bend left=15] (b1)
          edge [draw,cyan,bend right=10] (a2)
          edge [draw,green,bend right=6] (b2)
          edge [draw,gray,bend left=15,name path = diag] (b3);

  \fill (0,-.5) circle (2pt) node [below right] {$\ba_1$};
  \fill (0,2.5) circle (2pt) node [right,yshift = 2pt] {$\bb_1$};
  \fill (2.827,-.5) circle (2pt) node [below right] {$\bb_2$};
  \fill (-2.233,1) circle (2pt) node [below=3pt] {$\ba_2$};
  \fill (1.5,4.5) circle (2pt) node [above right] {$\bb_3$};
  \fill (-1.25,3) circle (2pt) node [above left] {$\ba_3$};

  \path[name intersections={of=vert1 and diag}];
  \path (intersection-1) coordinate (int1);
  \path[name intersections={of=vert2 and diag}];
  \draw (intersection-1) coordinate (int2);
  \path[name intersections={of=vert3 and diag}];
  \path (intersection-1) coordinate (int3);
  \path[name intersections={of=vert4 and diag}];
  \draw (intersection-1) coordinate (int4);
  \coordinate [above = .15cm] (P) at (int1);
  \coordinate [above = .15cm] (Q) at (int2);
  \draw [gray] (int2) -- (Q) -- (P);
  \draw [gray] (int3) -- +(0,.07);
  \draw [gray] (int3) -- +(0,-.07);
  \draw [gray] (int4) -- +(0,.07);
  \draw [gray] (int4) -- +(0,-.07);

  \draw [line width=1pt] (0,2.5)--(0,5);
  \draw [line width=1pt] (0,-1.65)--(0,-3) node [pos=.6,right] {\normalsize$L_1$};
  \path[name intersections={of=vert5 and ell}];
  \draw (intersection-1) node [below] {$L_2$};

 \end{tikzpicture}

 \caption{The octahedron}\label{fig_octahedron}

\end{figure}
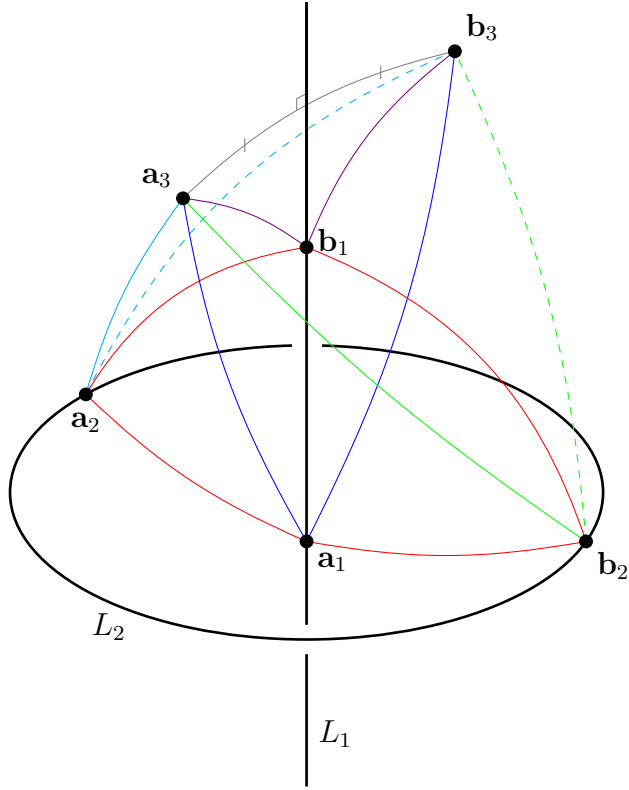

Rotating the octahedron $\ba_1\ba_2\ba_3\bb_1\bb_2\bb_3$ around the lines~$L_1$ and~$L_2$, we may always achieve that the vertex~$\ba_3$ lies on the arc given by $x_2=x_4=0$, $x_1>0$, $x_3>0$. Then
\begin{align*}
 \ba_1&=(\cos\alpha_1,\sin\alpha_1,0,0),&
 \bb_1&=(\cos \alpha_2,\sin \alpha_2,0,0),\\
 \ba_2&=(0,0,\cos \beta_1,\sin \beta_1),&
 \bb_2&=(0,0,\cos \beta_2,\sin \beta_2),\\
 \ba_3&=(\cos\theta,0,\sin\theta,0),&
 \bb_3&=(\cos\theta,0,-\sin\theta,0)
\end{align*}
for some $\alpha_1,\alpha_2,\beta_1,\beta_2\in \R$ and $\theta\in(0,\pi/2)$.
We denote by~$\ell_{\bu\bv}$ the length of an edge~$\bu\bv$ of the octahedron.

Then we have
\begin{gather*}
\nonumber \ell_{\ba_1\ba_2}=\ell_{\ba_1\bb_2}=\ell_{\bb_1\ba_2}=\ell_{\bb_1\bb_2}=\frac{\pi}{2}\,,\\
\ell_{\ba_1\ba_3}=\ell_{\ba_1\bb_3},\qquad \ell_{\bb_1\ba_3}=\ell_{\bb_1\bb_3},\\
 \ell_{\ba_2\ba_3}+\ell_{\ba_2\bb_3}=\ell_{\bb_2\ba_3}+\ell_{\bb_2\bb_3}=\pi.
\nonumber
\end{gather*}
So all edge lengths of the octahedron are solely determined by the four edge lengths~$\ell_{\ba_1\ba_3}$, $\ell_{\bb_1\ba_3}$, $\ell_{\ba_2\ba_3}$, and~$\ell_{\bb_2\ba_3}$; we denote the \textit{cosines} of those four edge lengths by~$p_1$, $p_2$, $q_1$, and~$q_2$, respectively. Then we have
\begin{align*}
 p_1&=\langle\ba_1,\ba_3\rangle=\cos\alpha_1\cos\theta,&
 p_2&=\langle\bb_1,\ba_3\rangle=\cos\alpha_2\cos\theta,\\
 q_1&=\langle\ba_2,\ba_3\rangle=\cos\beta_1\sin\theta,&
 q_2&=\langle\bb_2,\ba_3\rangle=\cos\beta_2\sin\theta.
\end{align*}
Therefore,
\begin{equation}\label{eq_flex}
\begin{aligned}
 \ba_1&=\left(\frac{p_1}{\cos\theta},\delta_1\sqrt{1-\frac{p_1^2}{\cos^2\theta}},0,0\right),&&&
 \bb_1&=\left(\frac{p_2}{\cos\theta},\delta_2\sqrt{1-\frac{p_2^2}{\cos^2\theta}},0,0\right),\\
 \ba_2&=\left(0,0,\frac{q_1}{\sin\theta},\varepsilon_1\sqrt{1-\frac{q_1^2}{\sin^2\theta}}\right),&&&
 \bb_2&=\left(0,0,\frac{q_2}{\sin\theta},\varepsilon_2\sqrt{1-\frac{q_2^2}{\sin^2\theta}}\right),\\
  \ba_3&=(\cos\theta,0,\sin\theta,0),&&&
  \bb_3&=(\cos\theta,0,-\sin\theta,0),
\end{aligned}
\end{equation}
where the signs $\delta_1,\delta_2,\varepsilon_1,\varepsilon_2\in\{-1,1\}$ can be chosen arbitrarily. The obtained octahedron is shown in Fig.~\ref{fig_octahedron}. In this figure, the sphere~$\bS^3$ is shown using a stereographic projection chosen in such a way that $L_1$ becomes a straight line. The lengths of all red edges are~$\pi/2$. The full segments of each colour are equal to each other, and the full and dashed segments of each colour  complement each other up to~$\pi$.

We immediately obtain the following proposition.

\begin{propos}
 Suppose that
 \begin{equation*}
  \max\left\{p_1^2,p_2^2\right\}+\max\left\{q_1^2,q_2^2\right\}<1.
 \end{equation*}
Then formulae \eqref{eq_flex} provide a flexion of an octahedron in~$\bS^3$. The parameter~$\theta$ runs over~$[\theta_{\min},\theta_{\max}]$, where
\begin{align*}
 \theta_{\min}&=\arcsin\bigl(\max\{|q_1|,|q_2|\}\bigr),&
 \theta_{\max}&=\arccos\bigl(\max\{|p_1|,|p_2|\}\bigr).
\end{align*}
\end{propos}

Let us show that $|p_1|=|p_2|$ or $|q_1|=|q_2|$ lead to non-interesting examples of flexible octahedra. Suppose that $|p_1|=|p_2|$. Then we have two possibilities. If $\delta_1\sgn(p_1)=\delta_2\sgn(p_2)$, then $\ba_1=\pm\bb_1$, so the octahedron is inessential.  Assume that $\delta_1\sgn(p_1)=-\delta_2\sgn(p_2)$. Replace the vertex~$\bb_1$ with~$\bb_1'=\delta_1\delta_2\bb_1$ and the vertex~$\bb_3$ with $\bb_3'=-\bb_3$. Then the obtained octahedron $\ba_1\ba_2\ba_3\bb_1'\bb_2\bb_3'$ is symmetric with respect to the plane (great sphere)~$H$ given by $x_1=0$ so that $\ba_1$ is symmetric to~$\bb_1'$, $\ba_3$ is symmetric to~$\bb_3'$ and $\ba_2,\bb_2\in H$. Thus, we obtain a usual plane-symmetric Bricard octahedron.
Similarly, if $|q_1|=|q_2|$, then the octahedron is also either inessential or can be transformed to a plane-symmetric Bricard octahedron by replacing some of vertices with their antipodes. All these cases are not interesting for us.

In Section~\ref{section_alg_geo} we will prove that the constructed flexible octahedron~\eqref{eq_flex} is an exotic flexible octahedron (i.\,e. satisfies condition~(4) from Introduction), provided that $|p_1|\ne |p_2|$ and~$|q_1|\ne |q_2|$.

\begin{propos}\label{propos_antipode}
 Suppose that $\ba_1\ba_2\ba_3\bb_1\bb_2\bb_3$ is a flexible octahedron of the form~\eqref{eq_flex} with some parameters~$p_1$, $p_2$, $q_1$, and~$q_2$ such that $|p_1|\ne |p_2|$ and~$|q_1|\ne |q_2|$. Let $\ba_1'\ba_2'\ba_3'\bb_1'\bb_2'\bb_3'$ be a flexible octahedron obtained from $\ba_1\ba_2\ba_3\bb_1\bb_2\bb_3$ by replacing some of its vertices with their antipodes. Then, up to rotation of\/~$\bS^3$ and renaming the vertices, the new octahedron  has also the form~\eqref{eq_flex} with some parameters~$p_1'$, $p_2'$, $q_1'$, and~$q_2'$ such that $|p_1'|\ne |p_2'|$ and~$|q_1'|\ne |q_2'|$.
\end{propos}

\begin{proof}
 Replacing the vertices~$\ba_1$, $\bb_1$, $\ba_2$, and~$\bb_2$  with their antipodes only results in reversals of the signs of the numbers~$p_1$, $p_2$, $q_1$, and~$q_2$, respectively, so we again obtain a flexible octahedron of the desired form. Assume that we replace~$\bb_3$ with its antipode~$-\bb_3$. (The case of~$\ba_3$ is completely similar.) The points~$\ba_3$ and~$-\bb_3$ are symmetric to each other with respect to the line~$L_2$. Therefore, if we rename the vertices by~$\ba_1\leftrightarrow\ba_2$ and $\bb_1\leftrightarrow\bb_2$ and rotate~$\bS^3$ so as to swap~$L_1$ and~$L_2$, then we arrive at a flexible octahedron of the form~\eqref{eq_flex} with parameters $p_1'=q_1$, $p_2'=q_2$, $q_1'=p_1$, and~$q_2'=p_2$.
\end{proof}

\begin{remark}
 The construction described fits well with the general idea, going back to Bricard, that imposing certain symmetry conditions can yield flexible polyhedra, see~\cite{Sta00} and~\cite{Gai18} for further results in this direction. Note, however, that so far the idea of imposing symmetry conditions has only worked in dimensions $3$ and~$4$. It would be interesting to find out whether it is possible to construct flexible polyhedra of higher dimensions based on symmetry considerations.
\end{remark}

\section{Configuration spaces of flexible octahedra}
\label{section_Bricard}

Consider an octahedron~$P=\ba_1\ba_2\ba_3\bb_1\bb_2\bb_3$ in~$\bS^3$.
We choose and fix the set~$\bell=(\ell_{\bu\bv})$ of $12$ edge lengths, where $\{\bu,\bv\}$ runs over all (unordered) pairs of vertices, except for the three pairs~$\{\ba_i,\bb_i\}$, $i=1,2,3$. We always assume that the lengths of edges of each face~$\bu_1\bu_2\bu_3$ are the lengths of sides of a non-degenerate spherical triangle, that is, satisfy the inequalities $0<\ell_{\bu_i\bu_j}<\pi$, the strict triangle inequalities $\ell_{\bu_i\bu_j}<\ell_{\bu_i\bu_k}+\ell_{\bu_j\bu_k}$ and the inequality $\ell_{\bu_1\bu_2}+\ell_{\bu_2\bu_3}+\ell_{\bu_1\bu_3}<2\pi$.

Our goal is to study the configuration space~$\Sigma(\bell)$ of all octahedra~$P$ with the prescribed set of edge lengths~$\bell$, up to rotations of~$\bS^3$. As coordinates in this space, we conveniently take the tangents
$$
t_{\bu\bv}=\tan\left(\frac{\varphi_{\bu\bv}}{2}\right),
$$
where $\varphi_{\bu\bv}$ are the \textit{oriented dihedral angles} at the edges~$\bu\bv$ of the octahedron. Following~\cite[Sect.~2]{Gai15b}, we choose the signs of the oriented dihedral angles as follows. First, we orient the polyhedral surface~$P$ so that the orientation of the face~$\ba_1\ba_2\ba_3$ is given by the indicated order of its vertices. Second, suppose that~$\bu\bv$ is an edge of the octahedron and $\bu\bv\bw_1$ and~$\bu\bv\bw_2$ are the two faces that contain this edge. Choose a point~$\bp$ on the edge~$\bu\bv$. For $i=1,2$, let $\bm_i\in T_{\bp}\bS^3$ be the \textit{exterior unit normal vector} to the face~$\bu\bv\bw_i$, that is, the unit vector orthogonal to~$\bu\bv\bw_i$ such that the product of the orientation of the normal to~$\bu\bv\bw_i$ given by the vector~$\bm_i$ by the positive orientation of the face~$\bu\bv\bw_i$ itself yields the positive orientation of~$\bS^3$. Further, let $\bn_i\in T_{\bp}\bS^3$ be the unit vector tangent to the face~$\bu\bv\bw_i$ and pointing inside it. Choose a positive direction around the edge~$\bu\bv$ so that $\bn_1$ is obtained from~$\bm_1$ by the rotation through an angle~$\pi/2$ in the positive direction. By definition, the \textit{oriented dihedral angle}~$\varphi_{\bu\bv}$ is the rotation angle from~$\bn_1$ to~$\bn_2$ in this positive direction. The angle~$\varphi_{\bu\bv}$ is well defined modulo~$2\pi\Z$. It is easy to check that it is independent of the choice of the point~$\bp$ and also of which of the two faces adjacent to~$\bu\bv$ is denoted by~$\bu\bv\bw_1$.

The tangents~$t_{\bu\bv}$ take their values in~$\R\cup\{\infty\}$. We conveniently identify~$\R\cup\{\infty\}$ with~$\RP^1$ so that $t_{\bu\bv}$ is an affine coordinate on~$\RP^1$.

Now, let $\bu$ be an arbitrary vertex of the octahedron, and let $\bu\bv_1\bv_2$, $\bu\bv_2\bv_3$, $\bu\bv_3\bv_4$, and~$\bu\bv_4\bv_1$ be the four faces containing~$\bu$.  We denote by $\bS^2_{\bu}$ the unit sphere in the tangent space~$T_{\bu}\bS^3$. Let $\bv_1',\ldots,\bv_4'\in \bS^2_{\bu}$ be the unit tangent vectors to the arcs of great circles $\bu\bv_1,\ldots,\bu\bv_4$, respectively. The spherical quadrilateral in~$\bS^2_{\bu}$ consisting of the four edges~$\bv_1'\bv_2'$, $\bv_2'\bv_3'$, $\bv_3'\bv_4'$, and~$\bv_4'\bv_1'$ is called the \textit{link} of the vertex~$\bu$ of the octahedron. The lengths of these four edges of the link are equal to the following angles of faces of the octahedron, respectively:
\begin{align*}
 \alpha&=\angle \bv_1\bu\bv_2,&
 \beta&=\angle \bv_2\bu\bv_3,&
 \gamma&=\angle \bv_3\bu\bv_4,&
 \delta&=\angle \bv_4\bu\bv_1.
\end{align*}
The angles of the link are equal to the corresponding dihedral angles of the octahedron.

Hereafter, by a \textit{spherical quadrilateral} we mean an arbitrary quadruple of points~$\bc_1$, $\bc_2$, $\bc_3$,~$\bc_4$ on sphere~$\bS^2$, together with the shortest great-circle arcs~$\bc_i\bc_{i+1}$ connecting them in a cyclic order, with the sole condition that $\bc_{i+1}\ne \pm\bc_i$ for $i=1,2,3,4$ (where sums of indices are taken modulo~$4$). So the lengths of sides of a spherical quadrilateral always belong to~$(0,\pi)$. We impose no other non-degeneracy conditions.

The following proposition about spherical quadrilaterals is due to Bricard~\cite{Bri97}. (Bricard studied tetrahedral angles in $3$-space rather than spherical quadrilaterals. Nevertheless, these two objects are obviously equivalent.)

\begin{propos}\label{propos_Bricard}
 The tangents~$t_1=t_{\bu\bv_1}$ and~$t_2=t_{\bu\bv_2}$ satisfy the biquadratic equation
 \begin{equation}
  At_1^2t_2^2+Bt_1^2+2Ct_1t_2+Dt_2^2+E=0,\label{eq_Bricard}
 \end{equation}
 where
 \begin{gather}
  A=\cos\gamma-\cos(\alpha+\beta+\delta),\qquad
  B=\cos\gamma-\cos(\alpha-\beta+\delta),\nonumber\\
  D=\cos\gamma-\cos(\alpha+\beta-\delta),\qquad
  E=\cos\gamma-\cos(\alpha-\beta-\delta),\label{eq_Bricard_coef}\\
  C=-2\sin\beta\sin\delta.\nonumber
 \end{gather}
\end{propos}

\begin{remark}
 As we have already said, each variables~$t_i$ should be considered as an affine coordinate on~$\RP^1$. Therefore, it is more correct to introduce the projective coordinates $t_i=(X_i:Y_i)$ on~$\RP^1$ and write relation~\eqref{eq_Bricard} in the form
 \begin{equation}
  AX_1^2X_2^2+BX_1^2Y_2^2+2CX_1Y_1X_2Y_2+DY_1^2X_2^2+EY_1^2Y_2^2=0.\label{eq_Bricard_proj}
 \end{equation}
 With this agreement, the converse of Proposition~\ref{propos_Bricard} is also true.
\end{remark}

\begin{propos}\label{propos_Bricard_converse}
 Suppose that $\alpha,\beta,\gamma,\delta\in(0,\pi)$. Let $(t_1,t_2)\in\RP^1\times\RP^1$ be an arbitrary solution of the equation~\eqref{eq_Bricard_proj} with coefficients~\eqref{eq_Bricard_coef}. For $i=1,2$, let $\varphi_i\in \R/2\pi\Z$ be the angle such that\/ $\tan(\varphi_i/2)=t_i$.  Then there exists a spherical quadrilateral~$\bv_1'\bv_2'\bv_3'\bv_4'$ in~$\bS^2$ with consecutive edge lengths~$\alpha$, $\beta$, $\gamma$, and~$\delta$ and the oriented angles at~$\bv_1'$ and~$\bv_2'$ equal to~$\varphi_1$ and~$\varphi_2$, respectively.
\end{propos}

\begin{remark}
By Lemma~4.1 in~\cite{Izm17} numbers $\alpha,\beta,\gamma,\delta\in (0,\pi)$ can be realized as consecutive side lengths of a non-degenerate (i.\,e. not contained in a great circle) spherical quadrilateral if and only if they satisfy the inequalities $\alpha<\beta+\gamma+\delta<\alpha+2\pi$, as well as all inequalities obtained from them by permutations of the side lengths~$\alpha$, $\beta$, $\gamma$, and~$\delta$. If degenerate quadrilaterals are allowed, then these inequalities should be made non-strict. Nevertheless, we do not need to add such inequalities to the conditions of Proposition~\ref{propos_Bricard_converse}. The matter is that if at least one of these inequalities is not fulfilled, then the system of equations~\eqref{eq_Bricard_proj} has no solutions.
\end{remark}

Now, we can construct the configuration space of octahedra with the prescribed set of edge lengths~$\bell$ as follows. Choose a face of the octahedron, say~$\ba_1\ba_2\ba_3$, and let $t_1=t_{\ba_2\ba_3}$, $t_2=t_{\ba_3\ba_1}$, and $t_3=t_{\ba_1\ba_2}$ be the tangents of the halves of the oriented dihedral angles at edges of this face. Then each of the pairs~$(t_1,t_2)$, $(t_2,t_3)$, and~$(t_3,t_1)$ satisfies a biquadratic equation of the form~\eqref{eq_Bricard_proj}. Note that the angles of faces of the octahedron and hence the coefficients of these three equations are solely determined by the set~$\bell$ of edge lengths of the octahedron. We denote by~$\Sigma_{\ba_1\ba_2\ba_3}(\bell)$ the algebraic variety in~$(\RP^1)^3$ given by these three equations. This variety will be called the \textit{configuration space} of spherical octahedra with the prescribed set of edge lengths~$\bell$ \textit{with respect to the face}~$\ba_1\ba_2\ba_3$. Propositions~\ref{propos_Bricard} and~\ref{propos_Bricard_converse} immediately imply the following proposition, which is also essentially due to Bricard~\cite{Bri97}, cf.~\cite{Gai14c} for a high-dimensional analogue. (Bricard studied only flexible octahedra in~$\E^3$. Nevertheless, the spherical case is completely similar.)

\begin{propos}\label{propos_conf_just}
 Up to rotations of~$\bS^3$, spherical octahedra with the prescribed set of edge lengths~$\bell$ are in one-to-one correspondence with points of the variety~$\Sigma_{\ba_1\ba_2\ba_3}(\bell)$.
\end{propos}

Since any pair of coordinates~$(t_j,t_k)$ satisfies a nontrivial algebraic relation of the form~\eqref{eq_Bricard_proj}, we obtain the following proposition.

\begin{propos}
 Any irreducible component of the variety~$\Sigma_{\ba_1\ba_2\ba_3}(\bell)$ has dimension at most~$1$.
\end{propos}

A disadvantage of the above construction is that it depends on the choice of the face~$\ba_1\ba_2\ba_3$. To overcome this disadvantage, we give the following more symmetric construction of the configuration space. In total we have $12$ variables $$t_{\bu\bv}=(X_{\bu\bv}:Y_{\bu\bv})\in\RP^1$$ indexed by the edges of the octahedron. They satisfy the $24$ biquadratic equations of the form~\eqref{eq_Bricard_proj} indexed by the pairs of edges that lie in the same face. Let~$\Sigma(\bell)$ be the algebraic variety in~$(\RP^1)^{12}$ given by these $24$ equations. This variety will be called the \textit{configuration space} of spherical octahedra with the prescribed set of edge lengths~$\bell$.

For each face~$F$ of the octahedron, we have the natural projection $\pi_F\colon\Sigma(\bell)\to\Sigma_F(\bell)$, which is a regular map of algebraic varieties.

\begin{propos}\label{propos_birat}
 The projection $\pi_F\colon\Sigma(\bell)\to\Sigma_F(\bell)$ is a homeomorphism. Moreover, it is also a birational equivalence, i.\,e. the inverse map~$\pi_F^{-1}$ is rational.
\end{propos}

\begin{proof}
We may assume that $F=\ba_1\ba_2\ba_3$. The fact that $\pi_{\ba_1\ba_2\ba_3}$ is a homeomorphism follows immediately from Proposition~\ref{propos_conf_just}. To prove that the map~$\pi_{\ba_1\ba_2\ba_3}^{-1}$ is rational, we need to show that each variable~$t_{\bu\bv}$ is a rational function of~$t_1=t_{\ba_2\ba_3}$, $t_2=t_{\ba_3\ba_1}$, and~$t_3=t_{\ba_1\ba_2}$. This fact follows immediately from the same assertion for a spherical quadrilateral, which was proved by Izmestiev, see Proposition~2.7 and Section~7 in~\cite{Izm15}. Alternatively, one can show that $t_{\bu\bv}$ is a rational function of~$t_1$, $t_2$, and~$t_3$ by combining the following two facts:

 (1) If we fix the positions of the vertices~$\ba_1$, $\ba_2$, and~$\ba_3$, then the coordinates of the remaining three vertices~$\bb_1$, $\bb_2$, and~$\bb_3$ will become rational functions of~$t_1$, $t_2$, and~$t_3$, see~\cite[formula~(2.5)]{Gai14c}.

 (2) The function $e^{i\varphi_{\bu\bv}}$ and hence the function
 $$
 t_{\bu\bv}=\tan(\varphi_{\bu\bv}/2)=-i\cdot\frac{e^{i\varphi_{\bu\bv}}-1}{e^{i\varphi_{\bu\bv}}+1}
 $$
 on the configuration space~$\Sigma(\bell)$ are rational functions of the coordinates of vertices of the octahedron. The proof of this fact for octahedra in the Lobachevsky space can be found in~\cite[Lemma~9.2]{Gai15a}; the spherical case is completely similar.
\end{proof}

\begin{cor}
 Any irreducible component of the variety~$\Sigma(\bell)$ has dimension at most~$1$.
\end{cor}

From Proposition~\ref{propos_birat} it follows that in all matters concerning the configuration space $\Sigma(\bell)$ as a set of points (or as a topological space), or the decomposition of the variety~$\Sigma(\bell)$ into irreducible components, or the fields of rational functions of irreducible components of~$\Sigma(\bell)$, we can identify~$\Sigma(\bell)$ via the projection~$\pi_F$ with any of the varieties~$\Sigma_F(\bell)$. However, in everything concerning smoothness properties, we must carefully distinguish between the varieties~$\Sigma(\bell)$ and~$\Sigma_F(\bell)$. Indeed, for any one-dimensional irreducible component~$\Gamma$ of the variety~$\Sigma(\bell)$ and the corresponding one-dimensional irreducible component~$\Gamma_F$ of the variety~$\Sigma_F(\bell)$, the projection $\Gamma\to\Gamma_F$ is a partial normalization of the curve~$\Gamma_F$; therefore, a smooth point of~$\Gamma$ can be projected to a singular point of~$\Gamma_F$. The word `partial' means that the curve $\Gamma$ is not necessarily smooth in general, i.\,e., the projection~$\pi_F$ resolves only some (but not all) singularities of the curve~$\Gamma_F$. In fact, later in the case of interest to us concerning exotic flexible octahedra, we will prove the smoothness of the curve~$\Gamma$, but only at its real points (see Proposition~\ref{propos_conf_exotic}).

\begin{defin}
 A one-dimensional irreducible component~$\Gamma$ of~$\Sigma(\bell)$ is said to be \textit{essential} if the corresponding flexion of the octahedron is essential. In other words, $\Gamma$ is essential if and only if any of the variables~$t_{\bu\bv}$ is neither identically zero nor indentically~$\infty$ on~$\Gamma$.
\end{defin}

Suppose that $\Gamma$ is an essential one-dimensional irreducible component of~$\Sigma(\bell)$. From  Proposition~\ref{propos_birat} it follows that the field of rational functions~$\R(\Gamma)$ is generated over~$\R$ by~$t_1$, $t_2$, and~$t_3$. Now we would like to consider all possibilities how the minimal polynomial relation between~$t_1$ and~$t_2$ on~$\Gamma$ can look like. For this purpose we need to study the decomposition of the relation~\eqref{eq_Bricard} into irreducible factors. Note that, since any $t_j$ is neither identically zero nor indentically~$\infty$ on~$\Gamma$, we may study relation~\eqref{eq_Bricard} rather than~\eqref{eq_Bricard_proj}. Moreover, we can neglect factors of the form~$t_j$. Bricard's classification of possible factorizations of relation~\eqref{eq_Bricard} into irreducible factors, up to factors of the form~$t_j$, is as follows, see~\cite{Bri97} (in terminology we follow~\cite[Sect.~2.4]{Izm17}):
\smallskip

\textsl{Case 1.} A spherical quadrilateral is called an \textit{isogram} if pairs of opposite sides have equal lengths, i.\,e. $\alpha=\gamma$ and $\beta=\delta$, and an \textit{antiisogram} if lengths of opposite sides complement each other to~$\pi$, i.\,e. $\alpha+\gamma=\beta+\delta=\pi$. Then the left-hand side of relation~\eqref{eq_Bricard} is the product of two factors, either of degree~$1$ with respect to each of the variables~$t_1$ and~$t_2$, and $t_1$ and~$t_2$ are either directly or inversely proportional to each other on~$\Gamma$. Therefore, $\R(t_1)=\R(t_2)$.
\smallskip

\textsl{Case 2.} A spherical quadrilateral is called a \textit{deltoid} if it has two pairs of adjacent equal sides, say $\alpha=\delta$ and $\beta=\gamma$, and an \textit{antideltoid} if it has two pairs of adjacent sides complementing each other to~$\pi$, say $\alpha+\delta=\beta+\gamma=\pi$. In addition, we suppose that the quadrilateral is neither an isogram nor an antiisogram, that is $\alpha\ne \gamma$ and $\alpha\ne\pi-\gamma$. Then, possibly after factoring out~$t_1$, the relation~\eqref{eq_Bricard} becomes irreducible of degree~$1$ with respect to~$t_1$ and of degree~$2$ with respect to~$t_2$. Hence, on~$\Gamma$ we have that $t_1$ is a rational function of~$t_2$, but $t_2$ is not a rational function of~$t_1$. Therefore, $\R(t_1)\subset\R(t_2)$ and $|\R(t_2)/\R(t_1)|=2$.

\textsl{Case 3.} Suppose that the quadrilateral is neither an isogram, nor an antiisogram, nor a deltoid, nor an antideltoid. Then the relation~\eqref{eq_Bricard} is irreducible and the degree of either variable~$t_j$ in it is equal to~$2$. Hence, on~$\Gamma$ we have that $$|\R(t_1,t_2)/\R(t_1)|=|\R(t_1,t_2)/\R(t_2)|=2,\qquad \R(t_1)\ne\R(t_2).$$ Usually this case is divided into two subcases, which are called \textit{rational} and \textit{elliptic} depending of whether the curve~$\Gamma$ is rational or elliptic. Nevertheless, in this paper we do not need to distinguish between these two subcases.

\begin{propos}\label{propos_44}
 Suppose that an essential one-dimensional irreducible component~$\Gamma$ of\/~$\Sigma(\bell)$ satisfies condition~\textnormal{(4)} from Introduction with respect to the face~$\ba_1\ba_2\ba_3$. Then\/~$\Gamma$ satisfies the same condition~\textnormal{(4)} with respect to any other face of the octahedron.
\end{propos}

\begin{proof}
 By condition~(4) for~$\ba_1\ba_2\ba_3$ we have that
 \begin{gather}
 |\R(\Gamma)/\R(t_{\ba_1\ba_2})|=4,\label{eq_ext1}\\
 |\R(t_{\ba_1\ba_3})/\R(t_{\ba_1\ba_2})|=|\R(t_{\ba_2\ba_3})/\R(t_{\ba_1\ba_2})|=2.\label{eq_ext2}
 \end{gather}
 First, from~\eqref{eq_ext1} it follows immediately that the octahedron satisfies condition~(4) with respect to the face~$\ba_1\ba_2\bb_3$, since none of the other cases~(1a), (1b), (2a), (2b), (2c), and~(3) is possible. Second, from~\eqref{eq_ext2} and the above classification of flexible spherical quadrilaterals it follows that the link of the vertex~$\ba_1$ is either a deltoid with $\angle\ba_2\ba_1\ba_3=\angle\ba_2\ba_1\bb_3$ and $\angle\bb_2\ba_1\ba_3=\angle\bb_2\ba_1\bb_3$ or an antideltoid with $\angle\ba_2\ba_1\ba_3=\pi-\angle\ba_2\ba_1\bb_3$ and $\angle\bb_2\ba_1\ba_3=\pi-\angle\bb_2\ba_1\bb_3$. (Moreover, it is neither an isogram nor an antiisogram.) Hence, $|\R(t_{\ba_1\ba_3})/\R(t_{\ba_1\bb_2})|=2$. Therefore, $|\R(\Gamma)/\R(t_{\ba_1\bb_2})|=4$. Thus, the octahedron satisfies condition~(4) with respect to either of the faces~$\ba_1\bb_2\ba_3$ and~$\ba_1\bb_2\bb_3$. Repeating the same reasoning, we obtain that the octahedron satisfies condition~(4) with respect to each of its faces.
\end{proof}

\begin{defin}\label{defin_exotic}
 Assume that for some set of edge lengths~$\bell$, the configuration space~$\Sigma(\bell)$ contains an essential one-dimensional irreducible component $\Gamma$ that (possibly after permuting the indices $1$, $2$, $3$) satisfies condition~(4) from Introduction, i.\,e., conditions~\eqref{eq_ext1}, \eqref{eq_ext2}. Then the flex of an octahedron corresponding to passing along any connected component of the curve~$\Gamma$ will be called an \textit{exotic flexible octahedron}.
\end{defin}

From Proposition~\ref{propos_44} it follows that the property of a flexible octahedron being exotic does not depend on which of its faces is denoted by~$\ba_1\ba_2\ba_3$.

\section{Equivalence of algebraic and geometric constructions}\label{section_alg_geo}

The aim of this section is to prove that the algebraic definition of an exotic flexible octahedron from Section~\ref{section_Bricard} (cf.~\cite{Gai14c}) agrees with the geometric construction from Section~\ref{section_constr}.

Suppose that, for some $\bell$, the configuration variety~$\Sigma(\bell)$ contains an essential one-dimensional irreducible component~$\Gamma$. Recall that in~\cite{Gai14c} the author proved that $\Gamma$  satisfies exactly one of the conditions~(1a), (1b), (2a), (2b), (2c), (3), and~(4) from Introduction. The case~(4) corresponds to exotic flexible octahedra.

\begin{theorem}\label{thm_alg_geo}
The curve~$\Gamma$ satisfies condition~\textnormal{(4)} from Introduction if and only if, up to rotation of\/~$\bS^3$ and simultaneous renaming of vertices $\ba_1\leftrightarrow \ba_2$ and~$\bb_1\leftrightarrow\bb_2$, the corresponding flexion of the octahedron admits a parametrization of the form~\eqref{eq_flex} with $|p_1|\ne|p_2|$ and~$|q_1|\ne |q_2|$.
\end{theorem}

\begin{proof}
First, consider the flexion of an octahedron given by~\eqref{eq_flex}, where $|p_1|\ne|p_2|$ and~$|q_1|\ne |q_2|$. By Proposition~\ref{propos_44} it is enough to prove condition~(4) for some one face of the octahedron. Swapping $\ba_1\leftrightarrow\bb_1$, we can achieve that $|p_1|>|p_2|$. As above, we put $t_1=t_{\ba_2\ba_3}$, $t_2=t_{\ba_3\ba_1}$, and $t_3=t_{\ba_1\ba_2}$.

The spherical law of cosines for the faces~$\ba_1\ba_2\ba_3$, $\ba_1\ba_2\bb_3$, $\bb_1\ba_2\ba_3$, and~$\bb_1\ba_2\bb_3$ yields that
\begin{align*}
 \cos\angle \ba_1\ba_2\ba_3&=\cos\angle \ba_1\ba_2\bb_3=\frac{p_1}{\sqrt{1-q_1^2}}\,,\\
 \cos\angle \bb_1\ba_2\ba_3&=\cos\angle \bb_1\ba_2\bb_3=\frac{p_2}{\sqrt{1-q_1^2}}\,.
\end{align*}
So the link of the vertex~$\ba_2$ is a deltoid and neither an isogram nor an antiisogram. Hence $t_3$ is a rational function of~$t_1$ and $|\R(t_1)/\R(t_3)|=2$. Similarly, the link of~$\ba_1$ is an antideltoid and neither an isogram nor an antiisogram, and hence $t_3$ is a rational function of~$t_2$ and $|\R(t_2)/\R(t_3)|=2$.  Let us prove that $\R(t_1)\ne\R(t_2)$. Indeed, it follows from Bricard's classification of flexible spherical quadrilaterals (see Section~\ref{section_Bricard}) that the equality~$\R(t_1)=\R(t_2)$ can only take place when~$t_1$ and~$t_2$ are either directly or inversely proportional to each other during the flexion. Since $|p_1|>|p_2|$, at $\theta=\theta_{\max}=\arccos|p_1|$ we have that the four vectors $\ba_1$, $\ba_2$, $\bb_2$, $\ba_3$ are linearly dependent, but the four vectors $\ba_1$, $\bb_1$, $\ba_2$, $\ba_3$ are linearly independent. It follows that $t_2$ is either~$0$ or~$\infty$, but $t_1$ is neither~$0$ nor~$\infty$. Therefore, $t_1$ and~$t_2$ are neither directly nor inversely proportional to each other during the flexion, so $\R(t_1)\ne\R(t_2)$. Thus, we have that $|\R(\Gamma)/\R(t_1)|\ge 2$ and hence $|\R(\Gamma)/\R(t_3)|\ge 4$. Since any essential flexible octahedron satisfies one of the conditions~(1a), (1b), (2a), (2b), (2c), (3), and~(4) from Introduction, we conclude that $|\R(\Gamma)/\R(t_3)|=4$ and the flexible octahedron under consideration satisfies condition~(4).

Second, let us prove the `only if' part of the theorem. Assume that $\Gamma$ is an arbitrary essential one-dimensional irreducible component of a configuration variety~$\Sigma(\bell)$ such that $\Gamma$ satisfies condition~(4) from Introduction. Our goal is to prove that, up to rotation of~$\bS^3$ and simultaneous renaming $\ba_1\leftrightarrow\ba_2$ and~$\bb_1\leftrightarrow\bb_2$, the corresponding flexion admits a parametrization of the form~\eqref{eq_flex} with $|p_1|\ne |p_2|$ and~$|q_1|\ne |q_2|$.

The following property of exotic flexible octahedra is Lemma~9.9 in~\cite{Gai14c}:
\smallskip

\textit{Suppose that $\bn_1$, $\bn_2$, and~$\bn_3$ are the unit inner normal vectors to the sides~$\ba_2\ba_3$, $\ba_3\ba_1$, and~$\ba_1\ba_2$ of the face~$\ba_1\ba_2\ba_3$, respectively; then $$\langle\bn_1,\bn_2\rangle=\langle\bn_1,\bn_3\rangle\langle\bn_2,\bn_3\rangle.$$}
\smallskip

This property immediately implies that $\ell_{\ba_1\ba_2}=\pi/2$. Similarly, we have that $$\ell_{\ba_1\bb_2}=\ell_{\bb_1\ba_2}=\ell_{\bb_1\bb_2}=\frac{\pi}{2}.$$
Consider  the great circle~$L_1$ through~$\ba_1$ and~$\bb_1$ and  the great circle~$L_2$ through~$\ba_2$ and~$\bb_2$. Then there exist Euclidean coordinates~$x_1,x_2,x_3,x_4$ in~$\E^4\supset\bS^3$ such that $L_1$ is given by $x_3=x_4=0$ and  $L_2$ is given by $x_1=x_2=0$.

Condition~(4) implies that, if $i$ is either~$1$ or~$2$, then $t_3$ is a rational function of~$t_i$ but $t_i$ is not a rational function of~$t_3$. From Bricard's classification of flexible spherical quadrilaterals (see Section~\ref{section_Bricard}) it follows that the link of each of the vertices~$\ba_1$ and~$\ba_2$ is either a deltoid or an antideltoid, so
\begin{align*}
 \cos\angle\ba_2\ba_1\ba_3&=\nu_1\cos\angle\ba_2\ba_1\bb_3,&
 \cos\angle\bb_2\ba_1\ba_3&=\nu_1\cos\angle\bb_2\ba_1\bb_3,&
 \nu_1&=\pm 1,\\
 \cos\angle\ba_1\ba_2\ba_3&=\nu_2\cos\angle\ba_1\ba_2\bb_3,&
 \cos\angle\bb_1\ba_2\ba_3&=\nu_2\cos\angle\bb_1\ba_2\bb_3,&
 \nu_2&=\pm 1.
\end{align*}
The following assertion is a standard consequence of the spherical laws of sines and cosines:
\smallskip

\textit{Suppose that the lengths of sides of a spherical triangle are equal to~$a$, $b$, and~$\pi/2$ and the angles of the corners opposite to the sides~$a$ and~$b$ are equal to~$\alpha$ and~$\beta$, respectively. Then}
$$
\cos a = \frac{\cos\alpha\sin\beta}{\sqrt{1-\cos^2\alpha\cos^2\beta}}\,.
$$
Applying this formula to the faces of the octahedron, we obtain that
\begin{align*}
 \cos\ell_{\ba_2\ba_3}&=\nu_1\cos\ell_{\ba_2\bb_3},&
 \cos\ell_{\bb_2\ba_3}&=\nu_1\cos\ell_{\bb_2\bb_3},\\
 \cos\ell_{\ba_1\ba_3}&=\nu_2\cos\ell_{\ba_1\bb_3},&
 \cos\ell_{\bb_1\ba_3}&=\nu_2\cos\ell_{\bb_1\bb_3}.
\end{align*}

Assume that the octahedron is essential. Then $\ba_1\ne\pm\bb_1$, $\ba_2\ne\pm\bb_2$, so $\ba_1,\bb_1,\ba_2,\bb_2$ is a basis of~$\E^4$.
If $\nu_1=\nu_2=1$, then we see that $\langle\ba_3,\bc\rangle=\langle\bb_3,\bc\rangle$ for $\bc=\ba_1,\bb_1,\ba_2,\bb_2$. Hence, $\ba_3=\bb_3$. Therefore, the octahedron is inessential. Similarly, if $\nu_1=\nu_2=-1$, then we obtain that $\ba_3=-\bb_3$ and the octahedron is inessential, too.

Thus, for an essential octahedron, we have that $\nu_1=-\nu_2$. By swapping $\ba_1\leftrightarrow\ba_2$ and~$\bb_1\leftrightarrow\bb_2$, we can achieve that $\nu_1=-1$ and~$\nu_2=1$. Let $\ba_3'$ be the point symmetric to~$\ba_3$ about~$L_1$. Then $\langle\ba_3',\bc\rangle=\langle\bb_3,\bc\rangle$ for $\bc=\ba_1,\bb_1,\ba_2,\bb_2$ and hence $\bb_3=\ba_3'$. Now, we see that $\ba_1,\bb_1\in L_1$, $\ba_2,\bb_2\in L_2$, and the vertices~$\ba_3$ and~$\bb_3$ are symmetric to each other with respect to~$L_1$. In Section~\ref{section_constr} we have shown that, up to rotation of~$\bS^3$, any octahedron that satisfies these conditions admits a parametrization of the form~\eqref{eq_flex}. Moreover, if at least one of the two equalities $|p_1|=|p_2|$ and $|q_1|=|q_2|$ were true, then, up to replacing of some vertices with their antipodes, the octahedron would be plane-symmetric and hence it would fall into one of the cases~(2a), (2b), and~(2c) rather than into the case~(4). Thus, $|p_1|\ne |p_2|$ and~$|q_1|\ne|q_2|$.
\end{proof}

\section{Configuration spaces of exotic flexible octahedra}\label{section_config_exotic}

In this section we describe the configuration space~$\Sigma=\Sigma(\bell)$, where $\bell$ is the set of edge lengths of an exotic flexible octahedron. Note that renaming the vertices by $\ba_1\leftrightarrow\bb_1$ and~$\ba_2\leftrightarrow\bb_2$ we can achieve that $|p_1|<|p_2|$ and $|q_1|<|q_2|$. It will be convenient for us to adopt this specific convention, which is opposite to the convention $|p_1|>|p_2|$ used in the proof of Theorem~\ref{thm_alg_geo}. The reason is that we would like none of the dihedral angles at the edges $\ba_i\ba_j$ to become either zero or straight during the flexion. (Note that in the proof of Theorem~\ref{thm_alg_geo}, it was precisely important that the dihedral angle at the edge $\ba_1\ba_3$ becomes zero or straight.)

\begin{propos}\label{propos_conf_exotic}
 Let~$\bell=(\ell_{\bu\bv})$ be the set of edge lengths for an octahedron such that $0<\ell_{\bu\bv}<\pi$ for all edges~$\bu\bv$,
 \begin{gather}
\nonumber \ell_{\ba_1\ba_2}=\ell_{\ba_1\bb_2}=\ell_{\bb_1\ba_2}=\ell_{\bb_1\bb_2}=\frac{\pi}{2}\,,\\
\ell_{\ba_1\ba_3}=\ell_{\ba_1\bb_3},\qquad \ell_{\bb_1\ba_3}=\ell_{\bb_1\bb_3},\label{eq_lengths}\\
 \ell_{\ba_2\ba_3}+\ell_{\ba_2\bb_3}=\ell_{\bb_2\ba_3}+\ell_{\bb_2\bb_3}=\pi,
\nonumber
\end{gather}
and the cosines
\begin{align*}
 p_1&=\cos\ell_{\ba_1\ba_3},&p_2&=\cos\ell_{\bb_1\ba_3},&
 q_1&=\cos\ell_{\ba_2\ba_3},&q_2&=\cos\ell_{\bb_2\ba_3}
\end{align*}
satisfy $|p_1|<|p_2|$, $|q_1|<|q_2|$, and $p_2^2+q_2^2<1$.
Then the configuration space $\Sigma=\Sigma(\bell)$ has a unique essential one-dimensional irreducible component~$\Gamma$, which consists of two connected components~$\Gamma_+$ and~$\Gamma_-$. For either connected component~$\Gamma_{\pm}$, the corresponding flexion of the octahedron, up to rotation of~$\bS^3$, is given by formulae~\eqref{eq_flex}, where the signs~$\delta_1$ and~$\varepsilon_1$ remain unchanged during the flexion, the sign~$\delta_2$ changes when $\theta=\theta_{\max}=\arccos|p_2|$, and the sign~$\varepsilon_2$ changes when $\theta=\theta_{\min}=\arcsin|q_2|$. Thus, $\Gamma_{\pm}$  projects $4$-fold (at interior points) onto the segment~$[\theta_{\min},\theta_{\max}]$, see Fig.~\ref{fig_circle}. The component~$\Gamma_+$ corresponds to $\delta_1\varepsilon_1=1$ and the component~$\Gamma_-$ corresponds to $\delta_1\varepsilon_1=-1$. Moreover, all (real) points of\/~$\Gamma$ are smooth and all coordinates of vertices of the flexible octahedron~\eqref{eq_flex} are real analytic functions on~$\Gamma$.
\end{propos}

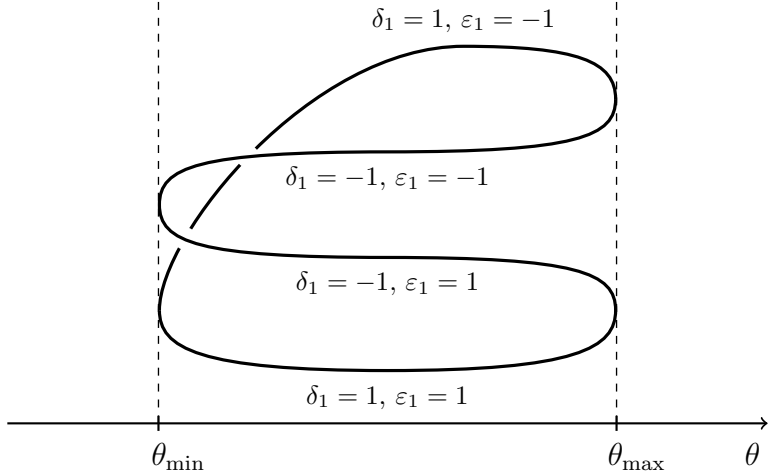
\begin{figure}
\begin{tikzpicture}
 \draw [line width=.8pt,->] (-2,0)--(8,0) node [below=3pt, pos = .225] {$\theta_{\min}$}
 node [below=3pt, pos = .828] {$\theta_{\max}$} node [below = 3pt, pos =.98] {$\theta$};
 \draw [line width=.5pt, dashed,xshift=-.35pt] (0,0) -- (0,5.6);
 \draw [line width=.5pt, dashed,xshift=.35pt] (6,0) -- (6,5.6);
 \draw [line width=.8pt,xshift=-.35pt]  (0,-.1)--(0,0.1);
 \draw [line width=.8pt,xshift=.35pt]  (6,-.1)--(6,0.1);

 \draw [line width =1.2pt] (6,4.3)
  .. controls (6,4.8) and (5.5,5) .. (4,5)
  node [above=1pt] {\footnotesize$\delta_2=1,\,\varepsilon_2=-1$}
  .. controls (2,5) and (0,2.5) .. (0,1.5);

  \fill [white,xshift=.34cm,yshift = 2.45cm] (0,0) circle (.15);
  \fill [white,xshift=1.16cm,yshift = 3.54cm] (0,0) circle (.15);
 \draw [line width =1.2pt] (0,1.5) .. controls (0,1) and (.5,.7) .. (3,.7)
  node [below=1pt] {\footnotesize$\delta_2=1,\,\varepsilon_2=1$}
  .. controls (5.5,.7) and (6,1) .. (6,1.5)
  .. controls (6,2) and (5.5,2.2) .. (3,2.2)
  node [below=1pt] {\footnotesize$\delta_2=-1,\,\varepsilon_2=1$}
  .. controls (.5,2.2) and (0,2.4) .. (0,2.9)
  .. controls (0,3.4) and (0.5,3.6) .. (3,3.6)
  node [below=1pt] {\footnotesize$\delta_2=-1,\,\varepsilon_2=-1$}
  .. controls (5.5, 3.6) and (6,3.8) .. (6, 4.3);
\end{tikzpicture}
 \caption{A connected component of the configuration space}
 \label{fig_circle}
\end{figure}

\begin{proof}
To prove this proposition, we need to establish a correspondence between the descriptions of  flexible octahedra in terms of the tangents of half dihedral angles and in terms of vertex coordinates.

When working with vertex coordinates, it is natural to consider an auxiliary configuration space~$\CC$ whose points are all possible positions of the flexible octahedron~\eqref{eq_flex} for various~$\theta$, $\delta_1$, $\delta_2$, $\varepsilon_1$, and~$\varepsilon_2$. Since the sign~$\delta_2$ is unimportant when $\theta=\theta_{\max}$ and the sign~$\varepsilon_2$ is unimportant when $\theta=\theta_{\min}$, we obtain that $\CC$ consists of the four connected components~$\CC_{\delta_1\varepsilon_1}$ corresponding to the choice of the signs~$\delta_1$ and~$\varepsilon_1$. Moreover, each of the connected components~$\CC_{\pm\pm}$ is homeomorphic to the circle and projects $4$-fold onto the segment~$[\theta_{\min},\theta_{\max}]$ as in Fig.~\ref{fig_circle}. Calculating the tangents~$t_{\bu\bv}$ of the half oriented dihedral angles for the octahedra~\eqref{eq_flex}, we obtain a map $f\colon\CC\to\Sigma$. We denote the image of~$f$ by~$\Gamma$. The rotation by~$\pi$ around the line $x_2=x_4=0$ changes simultaneously all the signs~$\delta_1$, $\delta_2$, $\varepsilon_1$, and~$\varepsilon_2$. Hence, the map~$f$ glues the components~$\CC_{++}$ and~$\CC_{--}$ together and glues the components~$\CC_{+-}$ and~$\CC_{-+}$ together.  We put $\Gamma_+=f(\CC_{++})=f(\CC_{--})$ and~$\Gamma_-=f(\CC_{+-})=f(\CC_{-+})$. To prove the proposition, we need to show that
\smallskip

(1) all points of~$\Sigma$ that correspond to essential octahedra lie in the image of~$f$,
\smallskip

(2) the restriction of~$f$ to~$\CC_{++}\cup\CC_{+-}$ is injective; in particular, $\Gamma_+\cap\Gamma_-=\varnothing$,
\smallskip

(3) $\Gamma$ is an irreducible component of~$\Sigma$,
\smallskip

(4) all real points of~$\Gamma$ are smooth and all coordinates of vertices in~\eqref{eq_flex} are real analytic functions on~$\Gamma$.
\smallskip

Assertion~(1) follows, since in Section~\ref{section_constr} we have shown that, up to rotation of~$\bS^3$, any essential octahedron with edge lengths~\eqref{eq_lengths} has the form~\eqref{eq_flex}.

To prove assertion~(2) we need to show that the parameters~$\theta$, $\delta_2$, $\varepsilon_1$, and $\varepsilon_2$ can be uniquely recovered from the tangents~$t_{\bu\bv}$ if we assume that $\delta_1=1$. We denote by~$y_1$, $y_2$, and~$y_3$ the cosines of the lengths of the diagonals~$\ba_1\bb_1$, $\ba_2\bb_2$, and~$\ba_3\bb_3$, respectively. Obviously, $y_1$, $y_2$, and~$y_3$ can be uniquely recovered from the tangents of the half dihedral angles of the octahedron. Further, we have
\begin{align}
 y_1&=\frac{p_1p_2}{\cos^2\theta}+\delta_2\sqrt{\left(1-\frac{p_1^2}{\cos^2\theta}\right)\left(1-\frac{p_2^2}{\cos^2\theta}\right)}\ ,\label{eq_diag1}\\
  y_2&=\frac{q_1q_2}{\sin^2\theta}+\varepsilon_1\varepsilon_2\sqrt{\left(1-\frac{q_1^2}{\sin^2\theta}\right)\left(1-\frac{q_2^2}{\sin^2\theta}\right)}\ ,\label{eq_diag2}\\
  y_3&=\cos^2\theta-\sin^2\theta.\label{eq_diag3}
\end{align}
Equality~\eqref{eq_diag3} allows us to recover~$\theta$. Then equality~\eqref{eq_diag1} allows us to recover~$\delta_2$ unless $\theta=\theta_{\max}$. However, if $\theta=\theta_{\max}$, then the sign~$\delta_2$ is unimportant. Similarly, equality~\eqref{eq_diag2} allows us to recover the product~$\varepsilon_1\varepsilon_2$ unless~$\theta=\theta_{\min}$, and the sign $\varepsilon_2$ is unimportant when $\theta=\theta_{\min}$. So if the images under~$f$ of two distinct points of~$\CC_{++}\cup\CC_{+-}$ coincide with each other, then these two points correspond to two rows of parameters $(\theta,\delta_2,\varepsilon_1,\varepsilon_2)$ that are obtained from each other by simultaneous reversing the signs~$\varepsilon_1$ and~$\varepsilon_2$. Finally, note that the symmetry about the plane $x_4=0$ results in reversing simultaneously the signs $\varepsilon_1$ and~$\varepsilon_2$. On the other hand, this symmetry reverses the sign of every tangent~$t_{\bu\bv}$. For each octahedron of the form~\eqref{eq_flex}, at least one of the tangents~$t_{\bu\bv}$ is neither~$0$ nor~$\infty$. Therefore, the images under~$f$ of the two points of~$\CC_{++}\cup\CC_{+-}$ corresponding to the rows of parameters $(\theta,\delta_2,\varepsilon_1,\varepsilon_2)$ and $(\theta,\delta_2,-\varepsilon_1,-\varepsilon_2)$ do not coincide with each other. Assertion~(2) follows.

Let us now prove assertion~(4). Recall that $\Gamma$ is a curve in the space~$(\RP^1)^{12}$ with coordinates~$t_{\bu\bv}$. Let $U\subset (\RP^1)^{12}$ be the Zariski open subset consisting of all points at which neither of the two coordinates $t_1=t_{\ba_2\ba_3}$ and $t_2=t_{\ba_1\ba_3}$ equals zero or infinity. It follows easily from the inequalities $|p_1|<|p_2|$ and $|q_1|<|q_2|$ that the tetrahedra~$\ba_1\bb_1\ba_2\ba_3$ and~$\ba_1\ba_2\bb_2\ba_3$ remain non-degenerate during the whole flexion~\eqref{eq_flex}. Hence, the curve~$\Gamma$ is contained in~$U$. Consider the dihedral angle between the faces~$\ba_1\ba_2\ba_3$ and~$\bb_1\ba_2\ba_3$. Applying standard formulae of spherical trigonometry, one can easily show that
$$
y_1=\cos h_1\cos h_2\cos z+\sin h_1\sin h_2\cdot\frac{1-t_1^2}{1+t_1^2}\,,
$$
where $h_1$ and $h_2$ are the lengths of the altitudes in the triangles $\ba_1\ba_2\ba_3$ and~$\bb_1\ba_2\ba_3$, respectively, dropped to their common side~$\ba_2\ba_3$, and $z$ is the distance between the feet of these altitudes. It follows that $y_1$ is a regular function on~$U$, and its differential~$dy_1$ does not vanish anywhere in~$U$. Similarly, $y_2$ is also a regular function on~$U$ whose differential is nonvanishing everywhere in~$U$.

It follows easily from~\eqref{eq_diag1} that, for any choice of the sign~$\delta_2$, the variable~$y_1$ is a real analytic function of the parameter~$\theta$ on the interval~$(\theta_{\min},\theta_{\max})$, and moreover, the derivative $dy_1/d\theta$ does not vanish anywhere in this interval. Hence, $\theta$ is a regular real analytic parameter on the curve~$\Gamma$ at all of its real points, except for the $8$ points where $\theta$ equals either~$\theta_{\min}$ or~$\theta_{\max}$. Therefore, all real point of~$\Gamma$ with $\theta\in (\theta_{\min},\theta_{\max})$ are smooth. Moreover, all coordinates of vertices in formulae~\eqref{eq_flex} are real analytic functions of~$\theta$ for~$\theta\in(\theta_{\min},\theta_{\max})$.

Let us now consider a point $\gamma\in\Gamma$ with $\theta=\theta_{\max}$. A convenient local parameter on~$\Gamma$ at~$\gamma$ is the function
$$
\tau=\delta_2\sqrt{1-\frac{p_2^2}{\cos^2\theta}}\ .
$$
Formula~\eqref{eq_diag1} reads as
$$
y_1=\frac{p_1}{p_2}\left(1-\tau^2\right)+\frac{\tau}{|p_2|}\sqrt{p_2^2-p_1^2+p_1^2\tau^2}\,.
$$
We see that $y_1$ is a real analytic function of the paratemer~$\tau$ in a neighborhood of the point $\tau=0$, and moreover, $\left.(dy_1/d\tau)\right|_{\tau=0}\ne0$. Hence,  $\gamma$ is a smooth point of the curve~$\Gamma$ and $\tau$ is a regular real analytic parameter on~$\Gamma$ in a neighborhood of~$\gamma$. Moreover, in a neighborhood of~$\gamma$, all coordinates of vertices in formulae~\eqref{eq_flex} are real analytic functions of~$\tau$.

The case of points where $\theta=\theta_{\min}$ is treated similarly, with~$y_1$ replaced by~$y_2$.

Finally, let us prove assertion~(3). Since $\Gamma_+$ and~$\Gamma_-$ are connected and consist of smooth points, we see that either~$\Gamma_+$ and~$\Gamma_-$ are irreducible components of~$\Sigma$ or $\Gamma=\Gamma_+\cup\Gamma_-$ is an irreducible component of~$\Sigma$. The fact that $\Gamma_+$ and~$\Gamma_-$ lie in the same irreducible component of~$\Sigma$ follows easily from the fact that the sign~$\delta_1\varepsilon_1$ will reverse as a result of analytic continuation with respect to the variable~$\theta$ along a loop~$\eta$ in~$\C$ that passes once around the point $\arccos |p_1|$, see Fig.~\ref{fig_gamma}.
\end{proof}

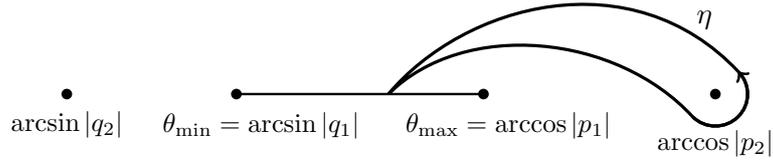
\begin{figure}
\begin{tikzpicture}\footnotesize

 \fill (.77,0) circle (2pt) node [below = 3pt] {$\arcsin|q_1|$};
 \fill (3,0) circle (2pt);
 \fill (6.25,0) circle (2pt);
 \fill (9.3,0) circle (2pt) node [below=10pt] {$\arccos|p_1|$};

 \draw [line width=.9pt] (3,0) -- (6.25,0) node [below=3pt, pos = 0.1] {$\theta_{\min}=\arcsin|q_2|$}
 node [below=3pt, pos = 1.1] {$\theta_{\max}=\arccos|p_2|$};

 \draw [line width=1.2pt] (5,0) .. controls (6,1) and (8,.8) .. (9,-.3) arc (-135:45:.424) .. controls (8.2,1.8) and (6,1.2) .. (5,0) node [pos = .1, above=1pt] {\small$\eta$} ;

 \draw [line width=1.2pt,->] (9,-.3) arc (-135:45:.424);

\end{tikzpicture}

 \caption{The loop $\eta$}\label{fig_gamma}
\end{figure}

\section{Calculation of volume}\label{section_volume}

The purpose of this section is to calculate the oriented volume of the exotic flexible octahedron~\eqref{eq_flex} with $|p_1|\ne |p_2|$ and $|q_1|\ne |q_2|$, and show that it is always nonconstant. Swapping the vertices $\ba_1\leftrightarrow\bb_1$ and~$\ba_2\leftrightarrow\bb_2$, we can achieve that $|p_1|<|p_2|$ and $|q_1|<|q_2|$. We assume these inequalities throughout this section. Let the configuration curve~$\Gamma$ be as in Proposition~\ref{propos_conf_exotic}.

We are going to write an explicit formula for the oriented volume~$\CV$ of the flexible octahedron~\eqref{eq_flex}.
Certainly, we would like to express $\CV$ through some parameter of the flexion. However, it turns out that $\theta$ is not a convenient parameter. A more convenient parameter is the cosine of the length of the diagonal~$\ba_1\bb_1$. In Section~\ref{section_config_exotic} we denoted this parameter by~$y_1$. We will now conveniently omit index~$1$ and denote it simply by~$y$.  We have
\begin{equation}\label{eq_y_theta}
 y=\langle\ba_1,\bb_1\rangle=\frac{p_1p_2}{\cos^2\theta}+\delta_1\delta_2\sqrt{
 \left(1-\frac{p_1^2}{\cos^2\theta}\right)
 \left(1-\frac{p_2^2}{\cos^2\theta}\right)
 }
\end{equation}

\begin{propos}
Let\/ $\Gamma_{\sigma}$ be any of the two connected components\/~$\Gamma_{\pm}$ of\/~$\Gamma$.  The largest and smallest values of~$y$ on~$\Gamma_{\sigma}$ are
 \begin{align*}
  y_{\max,\min}&=\frac{p_1p_2\pm\sqrt{(1-q_2^2-p_1^2)(1-q_2^2-p_2^2)}}{1-q_2^2},
 \end{align*}
where $y_{\max}$ and\/~$y_{\min}$  correspond to the signs~$+$ and~$-$, respectively. The value $y_{\max}$ is attained at the point $\gamma_{\max}\in\Gamma_{\sigma}$ with $\theta=\theta_{\min}$ and $\delta_1\delta_2=1$, and the value $y_{\min}$ is attained at the point $\gamma_{\min}\in\Gamma_{\sigma}$ with $\theta=\theta_{\min}$ and $\delta_1\delta_2=-1$. Moreover, the function~$y$ is strictly monotonic on each of the two arcs of\/~$\Gamma_{\sigma}$ connecting~$\gamma_{\min}$ and~$\gamma_{\max}$.
\end{propos}

\begin{proof}
Put $x=\cos^{-2}\theta$; then $x$ is a strictly increasing function of $\theta$. The values of~$x$ at the endpoints of the segment~$[\theta_{\min},\theta_{\max}]$ are  $x_{\min}=(1-q_2^2)^{-1}$ and $x_{\max}=p_2^{-2}$, respectively. Now, formula~\eqref{eq_y_theta} reads as
 $$
 y^2-2p_1p_2xy+(p_1^2+p_2^2)x-1=0.
 $$
 Let $Q$ be the hyperbola given by this equation. It is easy to check that the vertical line $x=x_{\max}$ is tangent to the left branch of the hyperbola~$Q$. Hence, the graph of the dependence between~$x$ and~$y$ is the segment of~$Q$ enclosed between the lines $x=x_{\min}$ and~$x=x_{\max}$, see Fig.~\ref{fig_hyperbola}.
 The proposition follows easily.
\end{proof}

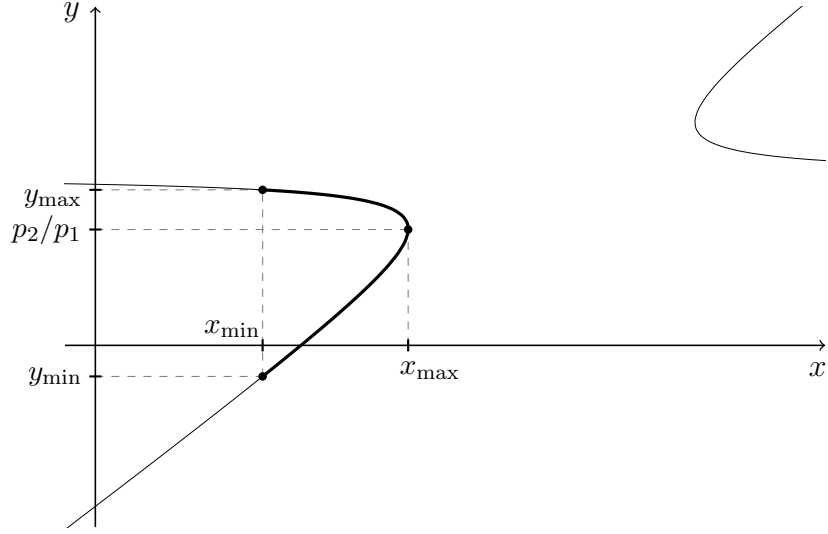
\begin{figure}

\begin{tikzpicture}[scale=2]
  \begin{scope}
    \draw [dashed,gray,line width = .4pt] (-3,-.0939)--(-1.9,-.0939);
    \draw [dashed,gray,line width = .4pt] (-3,-1.331)--(-1.9,-1.331);
    \draw [dashed,gray,line width = .4pt] (-3,-.357)--(-.943,-.357);
    \draw [dashed,gray,line width = .4pt] (-1.9,-1.331)--(-1.9, -.0939);
    \draw [dashed,gray,line width = .4pt] (-.943,-1.125)--(-.943, -.357);
    \clip (-3.2,-2.33) rectangle +(5,3.45);
    \pgfmathsetmacro{\a}{1}
    \pgfmathsetmacro{\b}{.333333333333}
    \draw[line width =.3pt] plot[domain=-2:2,smooth,rotate=18.434948823] ({\a*cosh(\x)},{\b*sinh(\x)});
    \draw[line width =.3pt] plot[domain=-2:2,smooth,rotate=18.434948823] ({-\a*cosh(\x)},{\b*sinh(\x)});

  \end{scope}
  \draw [line width=.6pt,->] (-3.2,-1.125)--(1.8,-1.125) node [below = 2pt, pos=.99] {$x$};
    \draw [line width=.6pt,->] (-3,-2.33)--(-3,1.12) node [left = 1.5pt, pos=.99] {$y$};
  \begin{scope}
   \clip (-1.9,-2.33) rectangle +(2,3.45);
   \pgfmathsetmacro{\a}{1}
    \pgfmathsetmacro{\b}{.333333333333}
   \draw[line width =1.2pt] plot[domain=-2:2,smooth,rotate=18.434948823] ({-\a*cosh(\x)},{\b*sinh(\x)});
  \end{scope}
   \fill (-.943,-.357) circle (.8pt);
   \fill (-1.9,-1.331) circle (.8pt);
   \fill (-1.9,-.0939) circle (.8pt);
   \draw [line width = .8pt] (-.943,-1.125+.04)-- (-.943,-1.125-.04);
   \draw (-.8,-1.125) node [below = 2pt] {\small$x_{\max}$};
   \draw [line width = .8pt] (-1.9,-1.125+.04)-- (-1.9,-1.125-.04);
   \draw (-2.1,-1.125) node [above=-1pt] {\small$x_{\min}$};
   \draw [line width = .8pt] (-3+.04,-1.331)-- (-3-.04,-1.331);
   \draw [line width = .8pt] (-3+.04,-.0939)-- (-3-.04,-.0939);
   \draw [line width = .8pt] (-3+.04,-.357)-- (-3-.04,-.357);
   \draw (-3,-1.331) node [left=1pt] {\small$y_{\min}$};
   \draw (-3,-.15) node [left=1pt] {\small$y_{\max}$};
   \draw (-3,-.357) node [left=1pt] {\small$p_1/p_2$};
\end{tikzpicture}

\caption{The hyperbola~$Q$} \label{fig_hyperbola}
\end{figure}

Now, we are ready to write an explicit formula for the oriented volume of the octahedron.
Since the oriented volume of a polyhedron in~$\bS^3$ is an element of~$\R/2\pi^2\Z$, all formulae for volume below will be modulo~$2\pi^2\Z$.

\begin{propos}\label{propos_volume}
 Suppose that $|p_1|<|p_2|$ and $|q_1|<|q_2|$. Put $r_j=(1-q_j^2)^{-1/2}$ for $j=1,2$. Then the oriented volume of the octahedron~\eqref{eq_flex} is given by the formula
 \begin{equation}\label{eq_volume}
  \CV=-\frac{\delta_1\sgn(p_2)\pi}{2}\bigl(\varepsilon_1A_1(y)-\varepsilon_2A_2(y)\bigr),
 \end{equation}
 where
 \begin{multline}
  A_j(y)=
  \arccos\frac{r_j(p_1-p_2y)}{\sqrt{(1-r_j^2p_2^2)(1-y^2)}}+\arccos\frac{r_j(p_2-p_1y)}{\sqrt{(1-r_j^2p_1^2)(1-y^2)}}\\
  {}+\arccos\frac{y-r_j^2p_1p_2}{\sqrt{(1-r_j^2p_1^2)(1-r_j^2p_2^2)}} -\pi.\label{eq_areas}
 \end{multline}
  Moreover, $0<A_1(y)<2\pi$ for all $y\in [y_{\min},y_{\max}]$,  $0<A_2(y)<2\pi$ for all $y\in (y_{\min},y_{\max})$, and $A_2(y_{\min})=A_2(y_{\max})=0$.
\end{propos}

Before proving this proposition, let us recall the following  formula for the area of a spherical triangle. It can be easily obtained by combining the standard formula $A=\alpha+\beta+\gamma-\pi$ through the spherical excess and the spherical law of cosines.
\smallskip

 \textit{Let $c_1$, $c_2$, and~$c_3$ be the cosines of lengths of sides of a spherical triangle. Then the area of this triangle is equal to}
 \begin{multline}\label{eq_ar_tri}
  A= \arccos\frac{c_1-c_2c_3}{\sqrt{(1-c_2^2)(1-c_3^2)}}+
     \arccos\frac{c_2-c_3c_1}{\sqrt{(1-c_3^2)(1-c_1^2)}}\\{}+
     \arccos\frac{c_3-c_1c_2}{\sqrt{(1-c_1^2)(1-c_2^2)}}-\pi.
 \end{multline}

\begin{proof}[Proof of Proposition~\ref{propos_volume}]
We denote the oriented and unoriented volumes of a spherical tetrahedron~$\bc_1\bc_2\bc_3\bc_4$ by $\Vor(\bc_1,\bc_2,\bc_3,\bc_4)$ and $V(\bc_1,\bc_2,\bc_3,\bc_4)$, respectively, and the area of a spherical triangle~$\bc_1\bc_2\bc_3$ by $A(\bc_1,\bc_2,\bc_3)$.

Decomposing the octahedron~\eqref{eq_flex} into four oriented spherical tetrahedra with common edge~$\ba_1\bb_1$, we obtain the following formula for the oriented volume~$\CV$ of it:
\begin{equation}\label{eq_vol_form1}
\begin{split}
\CV={}&\Vor(\bb_1,\ba_1,\ba_2,\ba_3)-\Vor(\bb_1,\ba_1,\ba_2,\bb_3)\\
{}-{}&\Vor(\bb_1,\ba_1,\bb_2,\ba_3)+\Vor(\bb_1,\ba_1,\bb_2,\bb_3).
\end{split}
\end{equation}
The rotation by~$\pi$ about the line~$L_1$ takes the tetrahedra~$\bb_1\ba_1\ba_2\bb_3$ and~$\bb_1\ba_1\bb_2\bb_3$ to the tetrahedra~$\bb_1\ba_1(-\ba_2)\ba_3$ and~$\bb_1\ba_1(-\bb_2)\ba_3$, respectively. Hence, formula~\eqref{eq_vol_form1} reads as
\begin{equation}\label{eq_vol_form2}
\begin{split}
\CV={}&\Vor(\bb_1,\ba_1,\ba_2,\ba_3)-\Vor(\bb_1,\ba_1,-\ba_2,\ba_3)\\
{}-{}&\Vor(\bb_1,\ba_1,\bb_2,\ba_3)+\Vor(\bb_1,\ba_1,-\bb_2,\ba_3).
\end{split}
\end{equation}
For any spherical tetrahedron~$\bc_1\bc_2\bc_3\bc_4$, we have that
$$
\Vor(\bc_1,\bc_2,\bc_3,\bc_4)=\sgn\bigl(\det(\bc_1,\bc_2,\bc_3,\bc_4)\bigr)V(\bc_1,\bc_2,\bc_3,\bc_4),
$$
where $\det(\bc_1,\bc_2,\bc_3,\bc_4)$ is the determinant  consisting of the coordinates of the vertices~$\bc_1$, $\bc_2$, $\bc_3$, and~$\bc_4$. It can be checked immediately that
\begin{align*}
 \sgn\bigl(\det(\bb_1,\ba_1,\ba_2,\ba_3)\bigr)&=-\delta_1\varepsilon_1\sgn(p_2),\\
 \sgn\bigl(\det(\bb_1,\ba_1,\bb_2,\ba_3)\bigr)&=-\delta_1\varepsilon_2\sgn(p_2).
\end{align*}
So formula~\eqref{eq_vol_form2} reads as
\begin{equation}\label{eq_vol_form3}
\begin{split}
\CV&=-\delta_1\varepsilon_1\sgn(p_2)\bigl(V(\bb_1,\ba_1,\ba_2,\ba_3)+V(\bb_1,\ba_1,-\ba_2,\ba_3)\bigr)\\
&\phantom{={}}
+\delta_1\varepsilon_2\sgn(p_2)\bigl(V(\bb_1,\ba_1,\bb_2,\ba_3)+V(\bb_1,\ba_1,-\bb_2,\ba_3)\bigr)\\
&=-\delta_1\varepsilon_1\sgn(p_2) \volume(\Phi_1)+\delta_1\varepsilon_2\sgn(p_2)\volume(\Phi_2),
\end{split}
\end{equation}
where $\Phi_1$ is the union of the two spherical tetrahedra~$\bb_1\ba_1\ba_2\ba_3$ and~$\bb_1\ba_1(-\ba_2)\ba_3$, and $\Phi_2$ is the union of the two spherical tetrahedra~$\bb_1\ba_1\bb_2\ba_3$ and~$\bb_1\ba_1(-\bb_2)\ba_3$.

Consider the points~$\pm\ba_2$ as poles of the sphere~$\bS^3$ and let~$H=\bS^2$ be the equatorial great sphere equidistant from them. Then $\ba_1$ and~$\bb_1$ lie on~$H$. Let $\bc$ be the intersection point of the meridian through~$\ba_3$ and the equatorial great sphere~$H$. The set~$\Phi_1$ is the union of the meridians through all points of the spherical triangle~$\ba_1\bb_1\bc$. Therefore,
\begin{equation*}
\volume(\Phi_1)=\frac{\volume(\bS^3)}{\area(\bS^2)}\cdot A(\ba_1,\bb_1,\bc)=\frac{\pi}{2}\cdot A(\ba_1,\bb_1,\bc).
\end{equation*}
Using the spherical law of cosines, we easily obtain that
$$
\cos\dist(\ba_1,\bc)=\cos\angle\ba_1\ba_2\ba_3=\frac{p_1}{\sqrt{1-q_1^2}}=r_1p_1.
$$
Similarly, we have $\cos\dist(\bb_1,\bc)=r_1p_2$. Substituting these values and also the value $\cos\dist(\ba_1,\bb_1)=y$ to~\eqref{eq_ar_tri}, we obtain that $A(\ba_1,\bb_1,\bc)=A_1(y)$ and hence $\volume(\Phi_1)=\pi A_1(y)/2$. Similarly, $\volume(\Phi_2)=\pi A_2(y)/2$. Thus, formula~\eqref{eq_vol_form3} reads as~\eqref{eq_volume}.

Since the tetrahedron~$\bb_1\ba_1\ba_2\ba_3$ is always non-degenerate, we see that $0<A_1(y)<2\pi$ for all~$y\in [y_{\min},y_{\max}]$. The tetrahedron~$\bb_1\ba_1\bb_2\ba_3$ is degenerate if and only if $y=y_{\min}$ or $y=y_{\max}$. Hence, $0<A_2(y)<2\pi$ for $y_{\min}<y<y_{\max}$ and $A_2(y_{\min})=A_2(y_{\max})= 0$.
\end{proof}

It is not hard to check that the functions~$A_1(y)$ and~$A_2(y)$ are real analytic on the interval~$(y_{\min},y_{\max})$, since none of the arccosines on the right-hand sides of the formulae~\eqref{eq_areas} for $j=1,2$ has singularities on this interval. Therefore, it follows from formula~\eqref{eq_volume} that the oriented volume~$\CV$ is a real analytic function at all points of~$\Gamma$, except for the four points where $\theta=\theta_{\min}$. Similarly, if we swap the roles of the vertices by $\ba_1\leftrightarrow\ba_2$ and $\bb_1\leftrightarrow\bb_2$, i.\,e. take  the cosine of the length of the diagonal~$\ba_2\bb_2$ for the parameter, then we will see that~$\CV$ is a real analytic function at all points of~$\Gamma$, except for the four points where $\theta=\theta_{\max}$. Thus, the function~$\CV$ is real analytic on the entire curve~$\Gamma$. Note, however, that this could also be proven without using the explicit formula for the volume. For the reader's convenience, we provide the corresponding argument.

\begin{propos}
The oriented volume of the flexible octahedron~\eqref{eq_flex} is a real analytic function $\CV\colon \Gamma\to\R/2\pi^2\Z$.
\end{propos}

\begin{proof}
Let $\omega$ be the standard volume form on~$\bS^3$. Choose a point $\bp\in\bS^3$ that does not belong to any face of the octahedron and a real analytic $2$-form $\psi$ on~$\bS^3\setminus\{\bp\}$ such that $d\psi=\omega|_{\bS^3\setminus\{\bp\}}$. Then by Stokes' theorem we have that the oriented volume of the octahedron is equal to the sum of the integrals of~$\psi$ over the faces of the octahedron. By Proposition~\ref{propos_conf_exotic} the coordinates of vertices of the octahedron are real analytic functions on~$\Gamma$. Since any face~$F$ of the octahedron remains non-degenerate during the flexion, we obtain that the integral of~$\psi$ over~$F$ is also a real analytic function.
\end{proof}

\begin{cor}\label{cor_nonconst}
 The oriented volume $\CV$ of the flexible octahedron~\eqref{eq_flex} with $|p_1|\ne |p_2|$ and $|q_1|\ne |q_2|$ is nonconstant on any arbitrarily small segment of the configuration curve~$\Gamma$.
\end{cor}

\begin{proof}
Consider the connected component $\Gamma_{\sigma}$ of~$\Gamma$, where $\sigma$ is either~$+$ or~$-$. By Proposition~\ref{propos_conf_exotic} we have the freedom to take $\delta_1=1$; then $\varepsilon_1=1$ if $\Gamma_{\sigma}=\Gamma_+$ and $\varepsilon_1=-1$ if $\Gamma_{\sigma}=\Gamma_-$. Choose an arbitrary $y_0\in(y_{\min},y_{\max})$ and let $\gamma_{\pm}\in\Gamma_{\sigma}$ be the two points with $y=y_0$ such that $\varepsilon_2=1$ at $\gamma_+$ and $\varepsilon_2=-1$ at $\gamma_-$. Then $0<A_2(y_0)<2\pi$. Hence,
$$\CV(\gamma_+)-\CV(\gamma_-)=\sgn(p_2)\pi A_2(y_0)\ne 0\pmod {2\pi^2\Z}.$$
Thus, the function~$\CV$ is nonconstant on~$\Gamma_{\sigma}$. Since $\CV$ is real analytic, it follows that $\CV$ is nonconstant on any arbitrarily small segment of~$\Gamma_{\sigma}$.
\end{proof}

Combining Corollary~\ref{cor_nonconst} and Proposition~\ref{propos_antipode}, we obtain the following result.

\begin{theorem}\label{thm_volume_precise}
Suppose that\/ $\ba_1\ba_2\ba_3\bb_1\bb_2\bb_3$ is a flexible octahedron of the form~\eqref{eq_flex} with $|p_1|\ne |p_2|$ and $|q_1|\ne |q_2|$, and\/ $\ba_1'\ba_2'\ba_3'\bb_1'\bb_2'\bb_3'$ is an octahedron obtained from it by replacing some of its vertices with their antipodes. Then the oriented volume of the octahedron~$\ba_1'\ba_2'\ba_3'\bb_1'\bb_2'\bb_3'$ is nonconstant on any arbitrarily small segment of the configuration curve~$\Gamma$. Thus, the octahedron~$\ba_1\ba_2\ba_3\bb_1\bb_2\bb_3$ provides a counterexample to the Modified Bellows Conjecture in\/~$\bS^3$.
\end{theorem}

\begin{remark}\label{remark_direct}
 In the above proof, it is important that the curve~$\Gamma$ is smooth at all its real points. Indeed, we have used the real analyticity of the function~$\CV$ to deduce its non-constancy on any arbitrarily small segment from the fact that it takes two different values. If the curve~$\Gamma$ had singular points, a situation could a priori occur where $\CV$ is constant on some (but not all) of the arcs into which these singular points divide the curve~$\Gamma$. Note, however, that the non-constancy of the function~$\CV$ on any arbitrarily small segment of~$\Gamma$ can also be proved without using the smoothness properties of the curve~$\Gamma$ and the analyticity of the function~$\CV$ by means of the following direct, though somewhat cumbersome, computation. (Thus, one obtains a proof of Theorem~\ref{thm_volume_precise} that does not use the results of Section~\ref{section_config_exotic}.) The author thanks the anonymous referee for pointing out this possibility.

 Differentiating equality~\eqref{eq_areas}, we obtain
 $$
 A_j'(y)=\frac{r_j(p_1+p_2)-y-1}{(y+1)\sqrt{F_j(y)}},
 $$
 where
 $$
 F_j(y)=-y^2+2r_j^2p_1p_2y-r_j^2(p_1^2+p_2^2)+1.
 $$
 To prove that the function~$\CV$ is non-constant on any arbitrarily small segment of~$\Gamma$, we suffice to show that the expression
 $$
 \bigl(A_1'(y)\bigr)^2-\bigl(A_2'(y)\bigr)^2=\bigl(A_1'(y)- A_2'(y)\bigr)\bigl(A_1'(y)+A_2'(y)\bigr)
 $$
 does not vanish identically on any subinterval of the segment~$[y_{\min},y_{\max}]$, i.\,e., that the polynomial
 $$
 Q(y)=\bigl(r_1(p_1+p_2)-y-1\bigr)^2F_2(y)-\bigl(r_2(p_1+p_2)-y-1\bigr)^2F_1(y)
 $$
 is not identically zero. The coefficient of~$y^3$ and the constant term of this polynomial are given by
 \begin{align*}
  c_3&=2(r_2-r_1)\bigl(p_1p_2(r_1+r_2)-p_1-p_2\bigr),\\
  c_0&=2(r_2-r_1)\bigl(r_1r_2(p_1+p_2)(p_1^2+p_2^2)-p_1p_2(r_1+r_2)+p_1+p_2\bigr),
 \end{align*}
 respectively. We have
 $$
 c_3+c_0=2(r_2-r_1)r_1r_2(p_1+p_2)(p_1^2+p_2^2)\ne 0.
 $$
 Therefore, the polynomial~$Q(y)$ is not identically zero.
\end{remark}

\section{Conclusion}\label{section_conclusion}

\subsection{On three kinds of exotic flexible octahedra} As we have already mentioned in Remark~\ref{remark_kind}, the kind of an exotic flexible octahedron of the form~\eqref{eq_flex} depends on the choice of a distinguished face. More precisely, we have the following result.

\begin{propos}\label{propos_kinds}
 Suppose that $|p_1|<|p_2|$ and $|q_1|<|q_2|$. Then the flexible octahedron~\eqref{eq_flex} has the first kind with respect to each of the faces~$\ba_1\ba_2\ba_3$ and\/~$\ba_1\ba_2\bb_3$, the second kind with respect to each of the faces~$\ba_1\bb_2\ba_3$, $\ba_1\bb_2\bb_3$, $\bb_1\ba_2\ba_3$, and\/~$\bb_1\ba_2\bb_3$, and the third kind with respect to each of the faces~$\bb_1\bb_2\ba_3$ and\/~$\bb_1\bb_2\bb_3$.
\end{propos}

\begin{proof}
 Consider a face $F=\bu_1\bu_2\bu_3$, where each $\bu_j$ is either~$\ba_j$ or~$\bb_j$. We would like to determine the kind of the flexible octahedron~\eqref{eq_flex} with respect to the face~$F$. To do this, we need to study the parameterizations for the tangents $t_1=t_{\bu_2\bu_3}$, $t_2=t_{\bu_3\bu_1}$, and $t_3=t_{\bu_1\bu_2}$ of the halves of the oriented dihedral angles in Jacobi's elliptic functions. From~\cite[Theorem~8.1]{Gai14c} (cf. Remark~\ref{remark_t}) it follows that, up to transformations $t_j\mapsto\pm t_j^{\pm 1}$ and swapping the indices~$1$ and~$2$, the tangents $t_1$, $t_2$, and~$t_3$ can be parametrized in one of the three forms~\eqref{eq_1kind}--\eqref{eq_3kind}. So for each of the indices $j=1$ and $j=2$, we have one of the following two possibilities:
 \begin{align}
  t_j &= \mu_j\dn^{\pm 1}(u-\sigma_j),\label{eq_t_param1}\\
  t_j &= \mu_j\left(\frac{\cn(u-\sigma_j)}{\sn(u-\sigma_j)}\right)^{\pm 1},\label{eq_t_param2}
 \end{align}
 where $\mu_j\ne 0$ and~$\sigma_j$ are real constants and the elliptic modulus satisfies $0<k<1$. The parameter of the flexion~$u$ runs over~$\R$.

 The octahedron with the distinguished face~$F$ has the first kind if we have parametrizations of the form~\eqref{eq_t_param1} for both $j=1$ and~$j=2$, the second kind if we have a parametrization of the form~\eqref{eq_t_param1} for one of the two indices and a parametrization of the form~\eqref{eq_t_param2} for the other, and the third kind if we have parametrizations of the form~\eqref{eq_t_param2} for both indices. To distinguish between the two possibilities~\eqref{eq_t_param1} and~\eqref{eq_t_param2}, we need the following properties of the functions $\dn u$ and~$\cn u /\sn u$ with the elliptic modulus $k\in (0,1)$, see~\cite[Section~13.18]{BaEr55}:
\begin{itemize}
 \item For $u\in\R$, the function $\dn u$ takes values on the segment~$[k',1]$, where $k'=\sqrt{1-k^2}$; in particular, it never takes the values~$0$ or~$\infty$.
 \item As $u$ runs over~$\R$, the function~$\cn u /\sn u$ takes all values in~$\R\cup\{\infty\}$.
\end{itemize}
So $t_1$ has a parametrization of the form~\eqref{eq_t_param2} if and only if the oriented dihedral angle $\varphi_{\bu_2\bu_3}$ becomes either zero or straight during the flexion and hence if and only if the tetrahedron~$\ba_1\bb_1\bu_2\bu_3$ becomes degenerate during the flexion. Similarly, $t_2$ has a parametrization of the form~\eqref{eq_t_param2} if and only if the tetrahedron~$\ba_2\bb_2\bu_1\bu_3$ becomes degenerate during the flexion.

Using formulae~\eqref{eq_flex} and the inequalities $|p_1|<|p_2|$ and $|q_1|<|q_2|$, it is easy to check that the tetrahedra~$\ba_1\bb_1\bb_2\ba_3$ and~$\ba_1\bb_1\bb_2\bb_3$ become degenerate when $\theta=\theta_{\min}=\arcsin|q_2|$, the tetrahedra~$\ba_2\bb_2\bb_1\ba_3$ and~$\ba_2\bb_2\bb_1\bb_3$ become degenerate when $\theta=\theta_{\max}=\arccos|p_2|$, and the tetrahedra~$\ba_1\bb_1\ba_2\ba_3$, $\ba_1\bb_1\ba_2\bb_3$, $\ba_2\bb_2\ba_1\ba_3$ and~$\ba_2\bb_2\ba_1\bb_3$ never become degenerate during the flexion. The proposition follows.
\end{proof}

\subsection{Higher-dimensional case} In~\cite{Gai14c} the author has also constructed exotic flexible cross-polytopes in~$\bS^n$, where $n\ge 4$. The tangents~$t_1,\ldots,t_n$ of the halves of the dihedral angles of these cross-polytopes at the $(n-2)$-dimensional faces contained in some $(n-1)$-dimensional face have parametrizations of the form
\begin{align*}
 t_j&=\lambda_j\dn u,&j\in I_1,\\
 t_j&=\lambda_j\dn \left(u-\frac{K}2\right),&j\in I_2,\\
 t_j&=\lambda_j\left(\dn u+\frac{k'}{\dn u}\right),&j\in I_3,
\end{align*}
where $\{1,\ldots,n\}=I_1\sqcup I_2\sqcup I_3$ is a partition into three non-empty subsets. (These formulae provide an analogue of exotic flexible octahedra of the first kind. Similarly, one can write formulae for the analogues of exotic flexible octahedra of the second and third kinds.)

In constrast to the three-dimensional case, in the higher-dimensional case no geometric construction of exotic flexible cross-polytopes is known. Note, however, that no geometric constructions are known for other types of flexible cross-polytopes in~$\E^n$, $\bS^n$, and $\Lambda^n$ either, except for some constructions by Walz (unpublished) and Stachel~\cite{Sta00} in dimension~$4$. The absence of a geometric description does not allow us to obtain an analogue of formula~\eqref{eq_volume} for the oriented volume of an exotic flexible cross-polytope. Nevertheless, the following conjecture seems quite plausible, although its proof may require cumbersome calculations with elliptic functions.

\begin{conj}
 Exotic flexible cross-polytopes provide counterexamples to the Modified Bellows Conjecture in\/~$\bS^n$ for all $n\ge 4$.
\end{conj}

Note also that from the analytic point of view the volume of a flexible polyhedron in even-dimensional Lobachevsky space behaves similarly to the volume of a flexible polyhedron in sphere, see~\cite{Gai15a}. Therefore, the incorrectness of the Modified Bellows Conjecture in~$\bS^3$ should be considered as an argument in favor of the fact that the Bellows Conjecture in even-dimensional Lobachevsky spaces may also be false.

\subsection{Real-valued branch of the volume}
Formula~\eqref{eq_volume} determines a single-valued real-valued branch of the oriented volume function for a flexible octahedron~\eqref{eq_flex}, that is, a lifting $\widetilde{\CV}\colon \Gamma\to\R$ of the function $\CV\colon\Gamma\to\R/2\pi^2\Z$. This means that the increment of the oriented volume along each of the connected components~$\Gamma_{\pm}$ is zero. The author does not know the answer to the following question.

\begin{question}
Does there exist a flexible polyhedron in the sphere~$\bS^3$ (or in the sphere~$\bS^n$ with $n>3$) such that the increment of its oriented volume along some closed curve in the corresponding configuration space is a nonzero multiple of the volume of the entire sphere, equivalently, such that its oriented volume function does not have a single-valued real-valued branch?
\end{question}

\medskip

\textbf{Acknowledgements.} The author is grateful to the anonymous referee for useful remarks.

\end{document}